\documentclass[11pt]{article}
\usepackage[english]{babel}		
\usepackage[T1]{fontenc}	
\usepackage{a4wide}
\usepackage{amsmath,amsfonts,amssymb,amsthm,cancel,siunitx,calculator,calc,mathtools,empheq,latexsym}
\usepackage{eucal}
\usepackage{graphicx}
\usepackage[dvipsnames]{xcolor}
\usepackage{subfig,epsfig,tikz,float}	
\usepackage{hyperref}
\hypersetup{colorlinks=true, linkcolor=red, filecolor=blue, urlcolor=cyan}

\theoremstyle{plain}
\newtheorem{thm}{Theorem}
\newtheorem{lem}[thm]{Lemma}

\theoremstyle{definition}

\newtheorem{defi}[thm]{Definition}

\numberwithin{thm}{section}
\numberwithin{equation}{section}

\def\supp{\operatorname{supp}}
\def\esup{\operatornamewithlimits{ess\,sup}}

\allowdisplaybreaks

\title{Weighted inequalities involving iteration of two Hardy integral operators}

\author{%
Amiran Gogatishvili$^{1}$ \ and \ Tu\u{g}\c{c}e \"{U}nver$^{1,2*}$
}

\begin{document}

\date{}

\maketitle

\vspace{-0.5cm}

\begin{center}
{\footnotesize 
$^1$ Institute of Mathematics of the
 Czech Academy of Sciences,
 \v Zitn\'a~25,
 115~67 Praha~1,
 Czech Republic \\
$^2$Department of Mathematics,
Kirikkale University,
71450 Yahsihan, Kirikkale,
Turkey \\
*Corresponding author. E-mail: tugceunver@kku.edu.tr, \texttt{ORCiD:}{0000-0003-0414-8400}\\
Contributing author: gogatish@math.cas.cz, \texttt{ORCiD:}{0000-0003-3459-0355}
}
\end{center}

\bigskip
\noindent

\bigskip
\noindent
{\small{\bf ABSTRACT.}
Let $1\leq p <\infty$ and $0 < q,r < \infty$. We characterize validity of the inequality for the composition of the Hardy operator,
\begin{equation*}
\bigg(\int_a^b \bigg(\int_a^x \bigg(\int_a^t f(s)ds \bigg)^q u(t) dt \bigg)^{\frac{r}{q}} w(x) dx  \bigg)^{\frac{1}{r}} \leq C \bigg(\int_a^b f^p(x) v(x) dx \bigg)^{\frac{1}{p}} 
\end{equation*} 
for all non-negative measurable functions on $(a,b)$, $-\infty \leq a < b \leq \infty$. We construct a more straightforward discretization method than those previously presented in the literature, and we characterize this inequality in both discrete and continuous forms.
}

\medskip
\noindent



\baselineskip=\normalbaselineskip

\section{Introduction and the main results}

Let $-\infty\leq a<b \leq \infty$. Denote by $\mathfrak{M}^+(a,b)$ the set of all non-negative measurable functions on $(a,b)$ and $\mathfrak{M}^{\uparrow}(a,b)$ is the class of non-decreasing elements of $\mathfrak{M}^+(a,b)$. 

In operator theory, weighted inequalities involving operator composition may be found in a wide range of topics. Let $0< q,r <\infty$ and $1 \leq p <\infty$. The validity of inequalities
\begin{equation} \label{HC-ineq.}
\bigg(\int_0^{\infty} \bigg(\int_0^x \bigg(\int_t^{\infty} h(s) ds\bigg)^q u(t) dt \bigg)^{\frac{r}{q}} w(x) dx  \bigg)^{\frac{1}{r}} \leq C \bigg(\int_0^{\infty} h(x)^p v(x) dx \bigg)^{\frac{1}{p}}, 
\end{equation} 
and
\begin{equation} \label{HH-ineq.}
\bigg(\int_0^{\infty} \bigg(\int_0^x \bigg(\int_0^t h(s) ds\bigg)^q u(t) dt \bigg)^{\frac{r}{q}} w(x) dx  \bigg)^{\frac{1}{r}} \leq C \bigg(\int_0^{\infty} h(x)^p v(x) dx \bigg)^{\frac{1}{p}},
\end{equation} 
for all $h \in \mathfrak{M}^+(0,\infty)$ are crucial, because many classical inequalities can be reduced to them. For example, duality techniques reduce the embeddings between Lorentz-type spaces, Morrey-type spaces and Ces\'{a}ro-type spaces to the weighted iterated inequalities (see, e.g. \cite{GKPS,GMU-CopCes,GMU-cLMLM,Unver-Ces}).

Various approaches have been used to handle inequalities \eqref{HC-ineq.} and \eqref{HH-ineq.} resulting with conditions of different nature. Inequality \eqref{HC-ineq.} is investigated  thoroughly. Detailed information on the development and history of this inequality may be found in the recent paper \cite{GMPTU}.


Our goal in this paper is to characterize \eqref{HH-ineq.}. When $q=1$ using Fubini's Theorem, inequality \eqref{HH-ineq.} reduces to the weighted Hardy-type inequality involving kernel, that is,
\begin{equation} \label{kernel-ineq.}
\bigg(\int_0^{\infty} \bigg(\int_0^x \bigg(\int_s^x u(t) dt \bigg) h(s) ds \bigg)^{r} w(x) dx  \bigg)^{\frac{1}{r}} \leq C \bigg(\int_0^{\infty} h(x)^p v(x) dx \bigg)^{\frac{1}{p}}, \quad h \in \mathfrak{M}^{+}(0,\infty)
\end{equation} 
Inequality \eqref{kernel-ineq.} was completely characterized in \cite{MR-Saw89,BloKer91,Oinarov94,Step94} when $1\leq r, p < \infty$.  However, for a long period there was no adequate characterization in the case when $0 < r < 1 \leq p < \infty$. Several attempts have been made to tackle this case (see, e.g. \cite{Step94,Lai99,Prok2013,GogStep2013}), in some works necessary and sufficient conditions did not match, while in others characterization had a discrete form or involved auxiliary functions, hence it was not easily verifiable. Finally, in \cite{Krepela-ker} the missing integral conditions were provided. 

We should also mention that in \cite{GogStep-JMAA,GogStep2013}, using reduction techniques, Hardy inequality involving non-decreasing functions, that is,
\begin{equation}\label{monot.ineq}
\bigg(\int_0^{\infty} \bigg(\int_0^x f(s) u(s) ds \bigg)^q  w(x) dx  \bigg)^{\frac{1}{q}} \leq C \bigg(\int_0^{\infty} f(x)^p v(x) dx \bigg)^{\frac{1}{p}},  \quad f \in \mathfrak{M}^{\uparrow}(0,\infty)
\end{equation}
is reduced to inequality \eqref{kernel-ineq.}.  However, as we have already mentioned, at that point of time the characterizations of the reduced inequalities were not known. Combination of  Theorem~3.13, Theorem~3.18 and Corollary~3.2 from \cite{GogStep2013} provides a characterization of \eqref{monot.ineq} but the result is non-standard and it is hard to extract the characterization from the theorems. The earlier works on inequality \eqref{monot.ineq} can be found in \cite{Hein-Step,Myas-Per-Step,Goldman12, Goldman13}.


We would like to point out that the characterization of \eqref{monot.ineq} may be obtained directly from inequality  \eqref{HH-ineq.} without any further work (see, the proof of Theorem~\ref{Cor}). We can provide the characterization of inequality \eqref{monot.ineq} as a direct outcome of our main theorem (see, Theorem~\ref{T:main}); nevertheless, we would like to provide it here to integrate all relevant parameter choices into a single theorem for the reader's convenience (see, Theorem~\ref{Cor}). 

In the general cases \eqref{HH-ineq.} is characterized in \cite{GogMus-MIA} but the conditions are in a non-standard form. It was also considered in \cite{ProkStep}, but the conditions are not applicable because they involve auxiliary functions. The special case of the dual version of \eqref{HH-ineq.} which involves iteration of the Copson operators $\int_t^{\infty} h$ is treated in \cite{Mus-Pos} when $p=1$, using a combination of reduction techniques and discretization. Recently, in \cite{KrePick-CC}, a more complicated discretization method is used to establish a characterisation of the same inequality that involves iteration of the Copson operators and is restricted to non-degenerate weights, and the case $p=1$ is presented without a proof. In our approach the case $p >1$ is not separated from $p= 1$.

As one can see in Section~\ref{Disc.Char.}, discretization method transforms the inequality at hand equivalently to discrete inequalities that involve local characterizations of inequalities having low-order iterations. For the very reason our aim in this paper is to revisit inequality \eqref{HH-ineq.} on $(a,b)$ where $-\infty\leq a<b \leq \infty$.

Recently, in \cite{GMPTU}, with a new and simpler discretization technique requires neither parameter restrictions  nor non-degeneracy conditions, characterization of \eqref{HC-ineq.} is given. We adapt this approach to the specific demands of the inequality considered in this paper. 

Let $-\infty\leq a<b \leq \infty$ and a weight be a positive measurable function on $(a, b)$. The principal goal of this study is to determine the necessary and sufficient conditions on weights $u,v,w$ on $(a,b)$ for which
\begin{equation} \label{main-iterated}
\bigg(\int_a^b \bigg(\int_a^x \bigg(\int_a^t f(s)ds \bigg)^q u(t) dt \bigg)^{\frac{r}{q}} w(x) dx  \bigg)^{\frac{1}{r}} \leq C \bigg(\int_a^b f^p(x) v(x) dx \bigg)^{\frac{1}{p}}
\end{equation} 
holds for $f\in \mathfrak{M}^+(a,b)$, with exponents $1 \leq p < \infty$ and $0 < q,r < \infty$. It is worth noting that if $p<1$, inequality \eqref{main-iterated} only holds for trivial functions. 

Let us first go through the essential notations and conventions before we present our main results. The left and right sides of the inequality numbered by $(*)$ are denoted by LHS$(*)$ and RHS$(*)$, respectively. We put $0.\infty =  \infty/\infty= 0/0 =0$. The symbol $A \lesssim B$ means that there exists a constant $c>0$ such that $A \leq c B$ where $c$ depends only on the parameters $p,q,r$. If both $A \lesssim B$ and $B \lesssim A$, then we write $A \approx B$.

For  $1\leq p < \infty$,  and $x,y \in [a,b]$, denote by
\begin{align}\label{Vp} 
	V_p(x, y) := \left\{ \begin{array}{ccc}
		\big(\int_x^y v^{-\frac{1}{p-1}}\big)^{\frac{p-1}{p}}, & 1<p<\infty,  \\
			\esup\limits_{s \in (x, y)} v(s)^{-1}, & p=1. 
	\end{array}
	\right.
\end{align}

Now, we are ready to formulate our main result. 

\begin{thm}\label{T:main}
Let $1 \leq p < \infty$, $0 < q, r < \infty$ and let $u, v, w$ be weights on $(a,b)$.  Then inequality \eqref{main-iterated} holds for all $f \in \mathfrak{M}^+(a, b)$ if and only if 

\rm(i) $p \leq r $, $p\leq q$ and
\begin{equation}\label{C1}
	C_1 :=  \esup_{x \in (a, b)} \bigg(\int_x^b w(t)  \bigg(\int_x^t u \bigg)^{\frac{r}{q}} dt \bigg)^{\frac{1}{r}} V_p(a, x) < \infty.
\end{equation}
Moreover, the best constant $C$ in inequality \eqref{main-iterated} satisfies $C \approx C_1$.

\rm(ii) $r < p \leq q$,
\begin{equation*}
	C_2 :=\bigg(\int_a^b \bigg(\int_x^b w \bigg)^{\frac{r}{p-r}} w(x) \esup_{t \in (a, x)} \bigg(\int_t^{x} u \bigg)^{\frac{rp}{q(p-r)}} V_p(a, t)^{\frac{pr}{p-r}} dx \bigg)^{\frac{p-r}{pr}}  < \infty,
\end{equation*}
and
\begin{equation}\label{C3}
	C_3 := \bigg( \int_a^b \bigg(\int_x^b w(s) \bigg(\int_{x}^s u \bigg)^{\frac{r}{q}} ds \bigg)^{\frac{r}{p-r}} w(x)  \esup_{t \in (a, x)} \bigg(\int_t^{x} u \bigg)^{\frac{r}{q}} V_p(a, t)^{\frac{pr}{p-r}} dx \bigg)^{\frac{p-r}{pr}}  < \infty.
\end{equation}
Moreover, the best constant $C$ in inequality \eqref{main-iterated} satisfies $C \approx C_2 + C_3$.

\rm(iii) $q< p \leq r $, $C_1 < \infty$ and
\begin{equation*}
	C_4 := \sup_{x \in (a, b)} \bigg(\int_x^b w \bigg)^{\frac{1}{q}} \bigg( \int_a^x \bigg(\int_{t}^x u \bigg)^{\frac{q}{p-q}} u(t) V_p(a, t)^{\frac{pq}{p-q}} dt  \bigg)^{\frac{p-q}{pq}} < \infty,
\end{equation*}
where $C_1$ is defined in \eqref{C1}. Moreover, the best constant $C$ in inequality \eqref{main-iterated} satisfies $C \approx C_1 + C_4$.

\rm(iv) $r < p$, $q < p$, $C_3 < \infty$ and  
\begin{equation*}
	C_5 := \bigg( \int_a^b \bigg(\int_x^b w \bigg)^{\frac{r}{p-r}} w(x) \bigg( \int_a^x \bigg(\int_{t}^x u \bigg)^{\frac{q}{p-q}} u(t) V_p(a, t)^{\frac{pq}{p-q}} dt \bigg)^{\frac{r(p-q)}{q(p-r)}} dx \bigg)^{\frac{p-r}{pr}}
	< \infty,
\end{equation*}
where $C_3$ is defined in \eqref{C3}. Moreover, the best constant $C$ in inequality \eqref{main-iterated} satisfies $C \approx C_3 + C_5$.
\end{thm}

\begin{thm}\label{Cor}
Let $0 < p,q <\infty$ and $u,v,w$ be weights on $(a,b)$. Then inequality 
\begin{equation}\label{monot.ineq-ab}
\bigg(\int_a^b \bigg(\int_a^x f(s) u(s) ds \bigg)^q  w(x) dx  \bigg)^{\frac{1}{q}} \leq C \bigg(\int_a^b f(x)^p v(x) dx \bigg)^{\frac{1}{p}}, 
\end{equation}
holds for all $f \in \mathfrak{M}^{\uparrow}(a,b)$ if and only if 

\rm(i) $p \leq q $, $p\leq 1$ and
\begin{equation}\label{C:C1}
	\mathcal{C}_1 :=  \esup_{x \in (a, b)} \bigg(\int_x^b w(t)  \bigg(\int_x^t u \bigg)^q dt \bigg)^{\frac{1}{q}} \bigg(\int_x^b v \bigg)^{-\frac{1}{p}} < \infty.
\end{equation}
Moreover, the best constant $C$ in inequality \eqref{monot.ineq-ab} satisfies $C \approx \mathcal{C}_1$.

\rm(ii) $q < p \leq 1$,
\begin{equation*}
\mathcal{C}_2 :=\bigg(\int_a^b \bigg(\int_x^b w \bigg)^{\frac{q}{p-q}} w(x) \esup_{t \in (a, x)} \bigg(\int_t^{x} u \bigg)^{\frac{pq}{p-q}} \bigg(\int_t^b v \bigg)^{-\frac{q}{p-q}} dx \bigg)^{\frac{p-q}{pq}}  < \infty,
\end{equation*}
and
\begin{equation}\label{C:C3}
\mathcal{C}_3 := \bigg( \int_a^b \bigg(\int_x^b w(s) \bigg(\int_{x}^s u \bigg)^q ds \bigg)^{\frac{q}{p-q}} w(x)  \esup_{t \in (a, x)} \bigg(\int_t^{x} u \bigg)^q \bigg(\int_t^b v \bigg)^{-\frac{q}{p-q}} dx \bigg)^{\frac{p-q}{pq}}  < \infty.
\end{equation}
Moreover, the best constant $C$ in inequality \eqref{monot.ineq-ab} satisfies $C \approx \mathcal{C}_2 + \mathcal{C}_3$.

\rm(iii) $1 < p \leq q $, $C_1 < \infty$ and
\begin{equation*}
\mathcal{C}_4 := \sup_{x \in (a, b)} \bigg(\int_x^b w \bigg) \bigg( \int_a^x \bigg(\int_{t}^x u \bigg)^{\frac{1}{p-1}} u(t) \bigg(\int_t^b v \bigg)^{-\frac{1}{p-1}} dt  \bigg)^{\frac{p-1}{p}} < \infty,
\end{equation*}
where $\mathcal{C}_1$ is defined in \eqref{C:C1}. Moreover, the best constant $C$ in inequality \eqref{monot.ineq-ab} satisfies $C \approx \mathcal{C}_1 + \mathcal{C}_4$.

\rm(iv) $q < p$, $1 < p$, $\mathcal{C}_3 < \infty$ and  
\begin{equation*}
\mathcal{C}_5 := \bigg( \int_a^b \bigg(\int_x^b w \bigg)^{\frac{q}{p-q}} w(x) \bigg( \int_a^x \bigg(\int_{t}^x u \bigg)^{\frac{1}{p-1}} u(t) \bigg(\int_t^b v \bigg)^{-\frac{1}{p-1}} dt \bigg)^{\frac{q(p-1)}{p-q}} dx \bigg)^{\frac{p-q}{pq}} < \infty,
\end{equation*}
where $\mathcal{C}_3$ is defined in \eqref{C:C3}. Moreover, the best constant $C$ in inequality \eqref{monot.ineq-ab} satisfies $C \approx \mathcal{C}_3 + \mathcal{C}_5$.
\end{thm}

Proofs of Theorem~\ref{T:main} and Theorem~\ref{Cor} will be given in  Section~\ref{S:Proofs}.

\section{Preliminary Results}

In this section, we cover the foundations of discretization as well as several new results that will be employed often throughout the proof of the main theorem.

\begin{defi}
Let $N \in \mathbb{Z} \cup \{-\infty\}$, $M \in \mathbb{Z} \cup \{+\infty\}$, $N<M$, and $\{a_k\}_{k=N}^M$ be a sequence of positive numbers. We say that $\{a_k\}_{k=N}^M$ is \textit{geometrically decreasing} if
\begin{equation*}
\sup\bigg\{\frac{a_{k+1}}{a_k},\quad N\leq k \leq M\bigg\} < 1.
\end{equation*}
\end{defi}

\begin{lem}\cite{GogPick2003}\label{L:dec-equiv-1}
Let $\alpha >0$ and $n \in \mathbb{Z}\cup\{-\infty\}$.  If $\{\tau_k\}_{k=n}^{\infty}$ is a  geometrically decreasing sequence,  then
\begin{equation}
\sup_{n\leq k<\infty} \tau_k \sum_{i=n}^k a_i  \approx \sup_{n\leq k <\infty} \tau_k a_k \label{sup-sum}
\end{equation} 
\begin{equation}
\sum_{k=n}^{\infty} \tau_k  \bigg(\sum_{i=n}^k a_i \bigg)^{\alpha}  \approx \sum_{k=n}^{\infty} \tau_k a_k^{\alpha}, \label{sum-sum}
\end{equation}
and
\begin{equation}
\sum_{k=n}^{\infty} \tau_k   \sup_{n \leq i \leq k} a_i  \approx \sum_{k=n}^{\infty} \tau_k  a_k, \label{sum-sup}
\end{equation}
for all non-negative sequences $\{a_k\}_{k=n}^{\infty} $.
\end{lem}

\begin{lem} \label{dec-equiv}
Let $\alpha >0$ and $n \in \mathbb{Z}\cup\{-\infty\}$.  Assume that $\{x_k\}_{k=n}^{\infty} $ is a strictly increasing sequence. If $\{\tau_k\}_{k=n}^{\infty}$ is a  geometrically decreasing sequence,  then,
\begin{equation}
\sup_{n \leq k <\infty} \tau_k \bigg( \int_{x_{n-1}}^ {x_k} g \bigg) \approx \sup_{n \leq k <\infty} \tau_k \bigg( \int_{x_{k-1}}^{x_{k}} g\bigg), \label{dec-sup-sum}
\end{equation}
\begin{equation}
\sum_{k=n}^{\infty} \tau_k  \bigg( \int_{x_{n-1}}^{x_k} g \bigg)^{\alpha}  \approx \sum_{k=n}^{\infty} \tau_k  \bigg( \int_{x_{k-1}}^{x_k} g \bigg)^{\alpha}, \label{dec-sum-sum}
\end{equation}
and
\begin{equation}
\sum_{k=n}^{\infty} \tau_k   \esup_{s \in (x_{n-1}, x_k)}  g(s)  \approx \sum_{k=n}^{\infty} \tau_k  \esup_{s \in (x_{k-1}, x_{k})}  g(s), \label{dec-sum-sup}
\end{equation}
for all non-negative measurable $g$ on $(x_{n-1}, \infty)$.
\end{lem}
\begin{proof}
Assume that $\{x_k\}_{k=n}^{\infty} $ is a strictly increasing sequence. For each $n \in \mathbb{Z} \cup \{-\infty\}$, we can write 
$$
\int_{x_{n-1}}^{x_k} g = \sum_{i=n}^k \int_{x_{i-1}}^{x_i} g.
$$   
Then \eqref{dec-sup-sum} and \eqref{dec-sum-sum} are direct consequences of \eqref{sup-sum} and \eqref{sum-sum}, respectively. 

Similarly, for each $n \in \mathbb{Z} \cup \{-\infty\}$, we have
$$
\esup_{s \in (x_{n-1}, x_k)}  g(s) = \sup_{n \leq i \leq k} \esup_{s \in (x_{i-1}, x_i)}  g(s),
$$
so that applying \eqref{sum-sup}, we obtain \eqref{dec-sum-sup}.
\end{proof}

\begin{lem}
Let  $\alpha >0$ and $n \in \mathbb{Z}\cup\{-\infty\}$. Assume that $\{x_k\}_{k=n}^{\infty}$ is a strictly increasing sequence, $\{\tau_k\}_{k=n}^{\infty}$ is a geometrically decreasing sequence, and $\{\sigma_k\}_{k=n}^{\infty}$ is a positive non-decreasing sequence. Then 
\begin{equation}\label{3-sup-equiv}
\sup_{n+1 \leq k <\infty}  \tau_k  \sup_{n \leq i < k} \bigg(\int_{x_i}^{x_k} g \bigg)^{\alpha} \sigma_i \approx \sup_{n+1 \leq k <\infty} \tau_k  \bigg(\int_{x_{k-1}}^{x_k} g\bigg)^{\alpha} \sigma_{k-1}.
\end{equation}
and
\begin{equation}\label{3-sum-equiv}
\sum_{k=n+1}^{\infty} \tau_k  \sup_{n \leq i < k} \bigg(\int_{x_i}^{x_k} g\bigg)^{\alpha} \sigma_i \approx \sum_{k=n+1}^{\infty} \tau_k  \bigg(\int_{x_{k-1}}^{x_k} g\bigg)^{\alpha} \sigma_{k-1}.
\end{equation}
hold for all non-negative measurable $g$ on $(x_n, \infty)$.
\end{lem}

\begin{proof}
Let us start with the equivalency \eqref{3-sup-equiv}. Since $\{\tau_k\}_{k=n}^{\infty}$ is a geometrically decreasing sequence, interchanging supremum and \eqref{dec-sup-sum} give 
\begin{equation*}
LHS\eqref{3-sup-equiv} = \sup_{n \leq i <\infty} \sigma_i \sup_{i+1 \leq k <\infty} \tau_k \bigg(\int_{x_i}^{x_k} g\bigg)^{\alpha} \approx    \sup_{n \leq i<\infty} \sigma_i \sup_{i+1 \leq k<\infty} \tau_k \bigg(\int_{x_{k-1}}^{x_k} g\bigg)^{\alpha}.
\end{equation*}
Interchanging supremum once again and monotonicity of $\{\sigma_k\}_{k=n}^{\infty}$ results in 
\begin{equation*}
LHS\eqref{3-sup-equiv} \approx \sup_{n+1 \leq k<\infty} \tau_k  \bigg(\int_{x_{k-1}}^{x_k} g\bigg)^{\alpha} \sup_{n\leq i \leq k-1} \sigma_i = RHS\eqref{3-sup-equiv}. 
\end{equation*}

Let us now tackle \eqref{3-sum-equiv}. Monotonicity of $\{\sigma_k\}_{k=n}^{\infty}$ gives that
\begin{align*}
LHS\eqref{3-sum-equiv} \leq \sum_{k=n+1}^{\infty} \tau_k  \sup_{n\leq i < k} \bigg(\sum_{j=i}^{k-1}  \sigma_j^{\frac{1}{\alpha}}  \int_{x_j}^{x_{j+1}} g \bigg)^{\alpha} = \sum_{k=n+1}^{\infty} \tau_k  \bigg(\sum_{j=n}^{k-1}  \sigma_j^{\frac{1}{\alpha}}  \int_{x_j}^{x_{j+1}} g \bigg)^{\alpha}.
\end{align*}
Then, using \eqref{sum-sum}, we have the following upper estimate
\begin{align*}
LHS\eqref{3-sum-equiv} \leq \sum_{k=n+1}^{\infty} \tau_k  \bigg(\sum_{j=n+1}^{k}  \sigma_{j-1}^{\frac{1}{\alpha}} \int_{x_{j-1}}^{x_{j}} g \bigg)^{\alpha} \approx RHS\eqref{3-sum-equiv}.
\end{align*}
On the other hand, the reverse estimate is clear and the proof is complete. 
\end{proof}


Let $w$ be a non-negative measurable function on $(a,b)$, denote by
\begin{equation*}
    W^*(t) = \int_t^b w(s)  \,ds, \quad t\in [a,b].
\end{equation*}

\begin{defi}
Let $w$ be a non-negative measurable function on $(a,b)$. A strictly increasing sequence  $\{x_k\}_{k=N}^{\infty}\subset [a,b]$ is said to be a discretizing sequence of the function $ W^*$, if it satisfies $ W^*(x_k)  \approx 2^{-k}$, $N \leq k <  \infty $. If $N > -\infty$ then $x_{N} : =a$, otherwise  $x_{-\infty}:= \lim_{k\rightarrow -\infty} x_k   = a$.	
\end{defi}

It is worth noting that if $N=-\infty$, then $N+1$ is also $-\infty$.

\begin{lem}
Let $\alpha \geq 0$ and $N \in \mathbb{Z}\cup \{-\infty\}$. Assume that $w$ is a weight on $[a, b]$ and $\{x_k\}_{k=N}^{\infty}$ is a discretizing sequence of the function $W^*$. Then for any $n \colon N\leq n$, 
\begin{equation}\label{int.equiv}
\int_{x_n}^b W^*(x)^{\alpha} w(x) h(x)  dx \approx \sum_{k=n+1}^{\infty} 2^{-k(\alpha +1)} h(x_k)
\end{equation}
and
\begin{equation}\label{sup.equiv}
\esup_{x \in (x_n,b)}  W^*(x)^{\alpha} h(x) \approx \sup_{n+1\leq k} 2^{-k{\alpha}} h(x_k)
\end{equation}
hold for all non-negative and non-decreasing $h$ on $(a, b)$. 
\end{lem}
\begin{proof}
Let $\{x_k\}_{k=N}^{\infty}$ be a discretizing sequence of the function $ W^*$. Monotonicity of $h$ and properties of the discretizing sequence $\{x_k\}_{k=N}^{\infty}$ yield
\begin{align*}
LHS\eqref{int.equiv} &= \sum_{k=n+1}^{\infty} \int_{x_{k-1}}^{x_k} h(x)  W^*(x)^{\alpha} w(x) dx  \lesssim \sum_{k=n+1}^{\infty} h(x_k) \int_{x_{k-1}}^{x_k}  d\Big[-W^*(x)^{\alpha+1} \Big] \\
& \approx \sum_{k=n+1}^{\infty} 2^{-k(\alpha+1)} h(x_k) = RHS\eqref{int.equiv} , 
\end{align*}
and, conversely
\begin{align*}
LHS\eqref{int.equiv} & \geq \sum_{k=n+1}^{\infty} \int_{x_k}^{x_{k+1}} h(x) W^*(x)^{\alpha} w(x) dx  \gtrsim \sum_{k=n+1}^{\infty} h(x_k) \int_{x_k}^{x_{k+1}}  d\Big[-W^*(x)^{\alpha+1} \Big] \\
& \approx \sum_{k=n+1}^{\infty} 2^{-k(\alpha+1)}  h(x_k) = RHS\eqref{int.equiv}.
\end{align*}
Thus, \eqref{int.equiv} holds. 

On the other hand, similarly, 
\begin{align*}
LHS\eqref{sup.equiv} =	\sup_{n+1 \leq k < \infty} \esup_{x \in (x_{k-1},x_k)}  W^*(x)^{\alpha} h(x) \approx \sup_{n+1 \leq k < \infty} 2^{-k\alpha} \esup_{x \in (x_{k-1},x_k)}  h(x) = RHS\eqref{sup.equiv}
\end{align*}
holds.
\end{proof}

\section{Discrete Characterization}\label{Disc.Char.}

We begin this section by observing that inequality \eqref{main-iterated} is equivalent to two other discrete inequalities, and we present the characterization in discrete form, which is noteworthy on its own.

Let us start with the discretization of inequality \eqref{main-iterated}. To this end we need the following notations: first denote by $B(x_{k-1}, x_k)$ the best constant of weighted Hardy inequality, that is, \begin{equation}\label{B(a,b)}
	B(x_{k-1},x_k) := \sup_{h\in  \mathfrak{M}^+(x_{k-1},x_k)} \frac{\bigg(\int_{x_{k-1}}^{x_k} \bigg(\int_{x_{k-1}}^t h(s) ds \bigg)^q u(t) dt \bigg)^{\frac{1}{q}}}{\bigg(\int_{x_{k-1}}^{x_k} h(t)^p v(t) dt\bigg)^{\frac{1}{p}}}
\end{equation}
and using the classical characterizations of weighted Hardy inequalities (see, \cite{KufPerSam,SinStep}), we have 
for $1\leq p <\infty, 0<q < \infty$
\begin{equation}\label{B-char.}
B(x_{k-1}, x_k) \approx\begin{cases}
\displaystyle \esup\limits_{t \in (x_{k-1}, x_{k})} \bigg(\int_t^{x_{k}} u \bigg)^{\frac{1}{q}} V_p(x_{k-1}, t) &\text{if} \quad p \leq q,\\
\\
\displaystyle 
\bigg(\int_{x_{k-1}}^{x_{k}} \bigg(\int_t^{x_{k}} u \bigg)^{\frac{q}{p-q}} u(t) V_p(x_{k-1}, t)^{\frac{pq}{p-q}}  dt \bigg)^{\frac{p-q}{pq}}  &\text{if} \quad q <  p.        \end{cases}
\end{equation}

Furthermore, observe that $V_p(x_{k-1}, x_k)$, $N+1 \leq k$ defined in \eqref{Vp} can be expressed as
\begin{eqnarray} \label{Vp-def}
 \sup_{g \in  \mathfrak{M}^+(x_{k-1}, x_k)} \frac{\int_{x_{k-1}}^{x_k} g(t) dt} {\bigg(\int_{x_{k-1}}^{x_k} g(t)^p v(t) dt \bigg)^{\frac{1}{p}}} = V_p(x_{k-1}, x_k).
\end{eqnarray}

\begin{lem} \label{T:equiv.ineq.-0}
Let $1\leq   p <\infty$, $0 < q, r < \infty$ and let $u, v, w$ be weights on $(a,b)$.   Assume that $\{x_k\}_{k=N}^{\infty}\subset [a,b]$ is a discretizing sequence of the function $W^*$. Then, it is clear that there exists a positive constant $C$ such that inequality \eqref{main-iterated} holds for all $f \in \mathfrak{M}^+(a,b)$ if and only if there exist positive constants $\mathcal{C}'$ and $\mathcal{C}''$ such that
\begin{equation}\label{main-iterated-1}
\bigg( \sum_{k=N+1}^{\infty}  2^{-k} \bigg( \int_{x_{k-1}}^{x_k} \bigg(\int_{x_{k-1}}^t f \bigg)^q u(t) dt \bigg)^{\frac{r}{q}}  \bigg)^{\frac{1}{r}} \leq \mathcal{C}' \bigg( \sum_{k=N+1}^{\infty} \int_{x_{k-1}}^{x_k} f^p v\bigg)^{\frac{1}{p}}
\end{equation}	
and
\begin{equation} \label{main-iterated-2}
\bigg( \sum_{k=N+1}^{\infty}  2^{-k}   \bigg( \int_a^{x_k} f \bigg)^r \bigg( \int_{x_k}^{x_{k+1}}  u \bigg)^{\frac{r}{q}}  \bigg)^{\frac{1}{r}}  \leq \mathcal{C}'' \bigg(\sum_{k=N+1}^{\infty} \int_{x_{k-1}}^{x_k} f^p v \bigg)^{\frac{1}{p}}
\end{equation}	
hold for all $f \in \mathfrak{M}^+(a,b)$. Moreover, $C \approx \mathcal{C}' + \mathcal{C}'' $.
\end{lem}
\begin{proof}
    Let $\{x_k\}_{k=N}^{\infty}$ be the discretizing sequence of the function $W^*$. Applying \eqref{int.equiv} with $\alpha =0$, we have that
\begin{equation*}
LHS\eqref{main-iterated} \approx \bigg( \sum_{k=N+1}^{\infty}  2^{-k} \bigg(\int_a^{x_k} \bigg(\int_a^t f \bigg)^q u(t) dt \bigg)^{\frac{r}{q}} \bigg)^{\frac{1}{r}}.
\end{equation*}
Since, $\{2^{-k}\}$ is geometrically decreasing, using \eqref{dec-sum-sum}, we obtain that
\begin{align*}
LHS\eqref{main-iterated} & \approx \bigg(  \sum_{k=N+1}^{\infty}   2^{-k} \bigg( \int_{x_{k-1}}^{x_k} \bigg(\int_a^t f\bigg)^q u(t) dt \bigg)^{\frac{r}{q}} \bigg)^{\frac{1}{r}} \\
& \approx \bigg( \sum_{k=N+1}^{\infty}    2^{-k} \bigg( \int_{x_{k-1}}^{x_k} \bigg(\int_{x_{k-1}}^t f\bigg)^q u(t) dt \bigg)^{\frac{r}{q}} \bigg)^{\frac{1}{r}} \\
&\quad + \bigg(  \sum_{k=N+2}^{\infty}    2^{-k} \bigg(\int_a^{x_{k-1}} f \bigg)^r \bigg( \int_{x_{k-1}}^{x_k}  u\bigg)^{\frac{r}{q}} \bigg)^{\frac{1}{r}}.
\end{align*}	
Then, it is clear that there exists a positive constant $C$ such that inequality \eqref{main-iterated} holds for all $f \in \mathfrak{M}^+(a,b)$ if and only if there exist positive constants $\mathcal{C}'$ and $\mathcal{C}''$ such that 
\eqref{main-iterated-1} and \eqref{main-iterated-2}
hold for all $f \in \mathfrak{M}^+(a,b)$. Moreover, $C \approx \mathcal{C}' + \mathcal{C}'' $.
\end{proof}

\begin{lem}\label{T:equiv.ineq.-1}
Let $1\leq   p <\infty$, $0 < q, r < \infty$ and let $u, v, w$ be weights on $(a,b)$.   Assume that $\{x_k\}_{k=N}^{\infty}\subset [a,b]$ is a discretizing sequence of the function $W^*$. Then there exists a positive constant $\mathcal{C}'$ such that inequality \eqref{main-iterated-1}
holds for all $f \in \mathfrak{M}^+(a,b)$ if and only if there exist positive constants $C'$ such that
\begin{equation}\label{Bk-inequality}
\bigg(\sum_{k=N+1}^{\infty} 2^{-k} a_k^r B(x_{k-1},x_{k})^r \bigg)^{\frac{1}{r}} \leq C'  \bigg(\sum_{k=N+1}^{\infty} a_k^p \bigg)^{\frac{1}{p}},
\end{equation}
holds for every sequence of non-negative numbers $\{a_k\}_{k=N+1}^{\infty}$. Moreover the best constants $\mathcal{C}'$ and $C'$, respectively, in \eqref{main-iterated-1} and \eqref{Bk-inequality} satisfies $\mathcal{C}' \approx C'$.
\end{lem}	

\begin{proof}
 Assume that \eqref{main-iterated-1} holds. By the definition of $B(x_{k-1}, x_k)$, there exist non-negative measurable functions $h_k$, $N+1 \leq k$ on $(a, b)$ such that 
$$
\supp h_k \subset [x_{k-1}, x_k], \quad \int_{x_{k-1}}^{x_k} h_k^p v = 1, \quad  \bigg(\int_{x_{k-1}}^{x_k} \bigg(\int_{x_{k-1}}^t h_k \bigg)^q u(t) dt\bigg)^{\frac{1}{q}} \gtrsim B(x_{k-1}, x_k).
$$
Thus, inserting $h = \sum_{m=N+1}^{\infty}  a_m h_m$,  where $\{a_m\}_{m=N+1}^{\infty}$ is any sequence of non-negative numbers, into \eqref{main-iterated-1}, \eqref{Bk-inequality} follows. Moreover, $C' \lesssim \mathcal{C}'$

Conversely, observe first that for each $h\in \mathcal{M}^+(x_{k-1},x_k)$
\begin{equation*}
    B(x_{k-1},x_k) \geq \bigg(\int_{x_{k-1}}^{x_k} \bigg(\int_{x_{k-1}}^t h(s) ds \bigg)^q u(t) dt \bigg)^{\frac{1}{q}} \bigg(\int_{x_{k-1}}^{x_k} h(t)^p v(t) dt\bigg)^{-\frac{1}{p}}.
\end{equation*}
Then \eqref{main-iterated-1} follows by, inserting  $a_k = (\int_{x_{k-1}}^{x_k} h^pv)^{\frac{1}{p}}$ in \eqref{Bk-inequality} and $\mathcal{C}' \leq C'$. Further, we have $\mathcal{C}' \approx C'$.
\end{proof}

\begin{lem}\label{T:equiv.ineq.-2}
Let $1 \leq p <\infty$, $0 < q, r < \infty$ and let $u, v, w$ be weights on $(a,b)$.   Assume that $\{x_k\}_{k=N}^{\infty}\subset [a,b]$ is a discretizing sequence of the function $W^*$. Then there exists a positive constant $\mathcal{C}''$ such that inequality \eqref{main-iterated-2} holds for all $f \in \mathfrak{M}^+(a,b)$ if and only if there exist a positive constant $C''$ such that
\begin{equation}\label{Vp-inequality}
\bigg(\sum_{k=N+1}^{\infty} 2^{-k} \bigg(\int_{x_k}^{x_{k+1}} u \bigg)^{\frac{r}{q}} \bigg(\sum_{j=N+1}^k a_j\, V_p(x_{j-1}, x_j) \bigg)^r \, \bigg)^{\frac{1}{r}} \leq C'' \bigg(\sum_{k=N+1}^{\infty} a_k^p\bigg)^{\frac{1}{p}},
\end{equation}	
holds for every sequence of non-negative numbers $\{a_k\}_{k=N+1}^{\infty}$. Moreover the best constants $\mathcal{C}''$ and $C''$, respectively, in \eqref{main-iterated-2} and \eqref{Vp-inequality} satisfies $\mathcal{C}'' \approx C''$.
\end{lem}	

\begin{proof}
Suppose that  \eqref{main-iterated-2} holds for all $f \in \mathfrak{M}^+(a,b)$. 
Using \eqref{Vp-inequality}, there exist non-negative measurable functions $g_k$, $N+1 \leq k$ on $(x_{k-1}, x_k)$ such that 
$$
\supp g_k \subset [x_{k-1}, x_{k}], \quad \int_{x_{k-1}}^{x_{k}} g_k^p v = 1, \quad  \int_{x_{k-1}}^{x_{k}} g_k  \gtrsim V_p(x_{k-1}, x_k).
$$
Thus, inserting $f = \sum_{m=N+1}^{\infty}  a_m g_m$,   where $\{a_m\}_{m=N+1}^{\infty}$ is any sequence of non-negative numbers, into \eqref{main-iterated-2}, \eqref{Vp-inequality} follows. Moreover, $C'' \lesssim \mathcal{C}''$ holds. 

Conversely, taking $a_k = (\int_{x_{k-1}}^{x_k} f^pv)^{\frac{1}{p}}$ in \eqref{Vp-inequality} and using the estimate  
\[ V_p(x_{k-1}, x_k)\ge \int_{x_{k-1}}^{x_k} f(t) dt\bigg(\int_{x_{k-1}}^{x_k} f(t)^p v(t) dt \bigg)^{-\frac{1}{p}} \]
gives \eqref{main-iterated-2}. Additionally,  $\mathcal{C}'' \leq C''$ holds. Consequently  $\mathcal{C}'' \approx C''$ follows. 
\end{proof}

\begin{thm}\label{T:disc.char.(3.3)}
Let $1 \leq p < \infty$, $0 < q, r < \infty$ and let $u, v, w$ be weights on $(a,b)$. Let $\{x_k\}_{k=N+1}^{\infty}$ be the discretizing sequence of $W^*$. Then inequality \eqref{Bk-inequality} holds for every sequence of non-negative numbers $\{a_k\}_{k=N+1}^{\infty}$ if and only if 

\rm(i) $p \le r$, $p\leq q$  and 
\begin{equation*}
\mathcal{A}_1 :=\sup_{N+1\leq  k <\infty} 2^{-\frac{k}{r}}  \esup_{t \in (x_{k-1}, x_{k})} \bigg(\int_t^{x_{k}} u \bigg)^{\frac{1}{q}} V_p(x_{k-1}, t)  < \infty,
\end{equation*}

\rm(ii) $r < p\leq q$ and
\begin{equation*}
\mathcal{A}_2 :=\bigg(\sum_{k=N+1}^{\infty} 2^{-k\frac{p}{p-r}} \esup_{t \in (x_{k-1}, x_{k})} \bigg(\int_t^{x_{k}} u\bigg)^{\frac{pr}{q(p-r)}} 
V_p(x_{k-1}, t)^{\frac{pr}{p-r}} \bigg)^{\frac{p-r}{pr}} < \infty,
\end{equation*}

\rm(iii) $q< p \le r $ and
\begin{equation*}
\mathcal{A}_3 :=\sup_{N+1 \leq k <\infty} 2^{-\frac{k}{r}} \bigg(\int_{x_{k-1}}^{x_{k}} \bigg(\int_t^{x_{k}} u \bigg)^{\frac{q}{p-q}} u(t) V_p(x_{k-1}, t)^{\frac{pq}{p-q}}  dt \bigg)^{\frac{p-q}{pq}} < \infty,
\end{equation*}

\rm(iv) $r < p$, $q < p$ and  
\begin{equation*}
\mathcal{A}_4 :=\bigg( \sum_{k=N+1}^{\infty} 2^{-k\frac{p}{p-r}} \bigg( \int_{x_{k-1}}^{x_{k}} \bigg(\int_{t}^{x_{k}} u \bigg)^{\frac{q}{p-q}} u(t)  V_p(x_{k-1}, t)^{\frac{pq}{p-q}}  dt \bigg)^{\frac{r(p-q)}{q(p-r)}} \bigg)^{\frac{p-r}{pr}} < \infty.
\end{equation*}

Moreover, the best constant $C'$ in inequality \eqref{Bk-inequality} satisfies 
\begin{equation*}
C' =\begin{cases}
\mathcal{A}_1  &\text{in the case \textup{(i)},}\\
\mathcal{A}_2  &\text{in the case \textup{(ii)},}\\
\mathcal{A}_3 &\text{in the case \textup{(iii)},}\\
\mathcal{A}_4 &\text{in the case \textup{(iv)}}
        \end{cases}
\end{equation*}
\end{thm}

\begin{proof}
The result follows easily by combining \cite[Theorem~4.5]{GPU-JFA}) with suitable parameters and weights with \eqref{B-char.}. 
\end{proof} 
\begin{thm}\label{T:disc.char.(3.4)}
Let $1 \leq p < \infty$, $0 < q, r < \infty$ and let $u, v, w$ be weights on $(a,b)$. Let $\{x_k\}_{k=N+1}^{\infty}$ be the discretizing sequence of $W^*$. Then inequality \eqref{Vp-inequality} holds for every sequence of non-negative numbers $\{a_k\}_{k=N+1}^{\infty}$if and only if 

\rm(i) $p \le r$ and
\begin{equation}\label{B1*}
\mathcal{B}_1 :=  \sup_{N+1 \leq k <\infty} \bigg(\sum_{i=k}^{\infty} 2^{-i} \bigg(\int_{x_i}^{x_{i+1}} u \bigg)^{\frac{r}{q}} \bigg)^{\frac{1}{r}} V_p(a, x_k) < \infty,
\end{equation}

\rm(ii) $r < p$ and
\begin{equation}\label{B2*}
\mathcal{B}_2 := \bigg( \sum_{k=N+1}^{\infty} 2^{-k} \bigg(\int_{x_k}^{x_{k+1}} u \bigg)^{\frac{r}{q}} \bigg(\sum_{i=k }^{\infty} 2^{-i} \bigg(\int_{x_i}^{x_{i+1}} u \bigg)^{\frac{r}{q}}  \bigg)^{\frac{r}{p-r}} V_p(a, x_k)^{\frac{pr}{p-r}} \bigg)^{\frac{p-r}{pr}}< \infty,
\end{equation}

Moreover, the best constant $C'' $ in inequality \eqref{Vp-inequality} satisfies 
\begin{equation*}
C'' =\begin{cases}
\mathcal{B}_1  &\text{in the case \textup{(i)},}\\
\mathcal{B}_2  &\text{in the case \textup{(ii)}.}
        \end{cases}
\end{equation*}

\end{thm}

\begin{proof}
We will apply \cite[Theorem~4.6]{GPU-JFA} with suitable parameters  and weights $a_k=2^{-k}\big(\int_{x_k}^{x_{k+1}} u \big)^{\frac{r}{q}} $ and $b_k = V_p(x_{k-1},x_k)$. We need to treat the cases when $p>1$ and $p=1$, separately. Note that when $p>1$, since $\{x_k\}_{k=N+1}^{\infty}$ is the discretizing sequence of $W^*$, we have for each $k: N+1 \leq k$ that
\begin{align*}
\sum_{i=N+1}^{\infty} V_p(x_{i-1},x_i)^{\frac{p}{p-1}} = \sum_{i=N+1}^{\infty} \int_{x_{i-1}}^{x_i} v(s)^{-\frac{1}{p-1}} ds = V_p(a,x_k)^{\frac{p-1}{p}}. 
\end{align*}
Thus applying \cite[Theorem~4.6, (iv)]{GPU-JFA} when $1<p\leq r$ and \cite[Theorem~4.6, (iii)]{GPU-JFA} when $1<p, r<p$, the result follows. 
 
On the other hand, if $p=1$, for each $k: N+1 \leq k$,
\begin{equation}
\sup_{N+1\leq i \leq k} V_p(x_{i-1},x_i) =  \sup_{N+1\leq i \leq k} \esup_{s\in(x_{i-1},x_i)} v(s)^{-1} =  V_p(a,x_k).  
\end{equation}
Therefore, applying \cite[Theorem~4.6, (ii)]{GPU-JFA} when $r<p=1$, we have $C'' \approx \mathcal{B}_2$. Lastly, if $p=1\leq r$, then  applying \cite[Theorem~4.6, (i)]{GPU-JFA}, we have
\begin{equation*}
C''  \approx \sup_{N+1 \leq k <\infty} \bigg(\sum_{i=k}^{\infty} 2^{-i} \bigg(\int_{x_i}^{x_{i+1}} u \bigg)^{\frac{r}{q}} \bigg)^{\frac{1}{r}} V_p(x_{k-1}, x_k).
\end{equation*}
Finally, interchanging supremum yields that
\begin{align*}
C''&\approx \sup_{N+1 \leq k <\infty} V_p(x_{k-1}, x_k) \sup_{k\leq m<\infty} \bigg(\sum_{i=m}^{\infty} 2^{-i} \bigg(\int_{x_i}^{x_{i+1}} u \bigg)^{\frac{r}{q}} \bigg)^{\frac{1}{r}} \\
& = \sup_{N+1 \leq m <\infty} \bigg(\sum_{i=m}^{\infty} 2^{-i} \bigg(\int_{x_i}^{x_{i+1}} u \bigg)^{\frac{r}{q}} \bigg)^{\frac{1}{r}}  \sup_{N+1\leq k \leq  m} V_p(x_{k-1}, x_k) = \mathcal{B}_1.
\end{align*}

\end{proof}
Now, we are in position to formulate the discrete characterization of inequality \eqref{main-iterated}.

\begin{thm}\label{T:disc.char.}
Let $1 \leq p < \infty$, $0 < q, r < \infty$ and let $u, v, w$ be weights on $(a,b)$. Let $\{x_k\}_{k=N+1}^{\infty}$ be the discretizing sequence of $W^*$. Then inequality \eqref{main-iterated} holds for all $f \in \mathfrak{M}^+(a,b)$ if and only if 

\rm(i) $p \le r$, $p\leq q$ and $
\mathcal{A}_1 + \mathcal{B}_1 <\infty$,

\rm(ii) $r < p\leq q$ and
$\mathcal{A}_2 + \mathcal{B}_2<\infty$,

\rm(iii) $q< p \le r $ and $\mathcal{A}_3 + \mathcal{B}_1<\infty$,

\rm(iv) $r < p$, $q < p$ and $\mathcal{A}_4 +\mathcal{B}_2 <\infty$.

Moreover the best constant $C$ in \eqref{main-iterated} satisfies
\begin{equation*}
C =\begin{cases}
\mathcal{A}_1 +\mathcal{B}_1  &\text{in the case \textup{(i)},}\\
\mathcal{A}_2 +\mathcal{B}_2  &\text{in the case \textup{(ii)},}\\
\mathcal{A}_3+\mathcal{B}_1  &\text{in the case \textup{(iii)},}\\
\mathcal{A}_4 +\mathcal{B}_2  &\text{in the case \textup{(iv)}.}
\end{cases}
\end{equation*}
\end{thm}

\begin{proof}
According to Lemmas~[\ref{T:equiv.ineq.-0}-\ref{T:equiv.ineq.-2}], inequality \eqref{main-iterated} holds if and only if inequalities \eqref{Bk-inequality} and \eqref{Vp-inequality} hold. Moreover, the best constant $C$ in \eqref{main-iterated} satisfies $C\approx C' + C''$, where $C'$ and $C''$ are the best constants in the inequalities \eqref{Bk-inequality} and \eqref{Vp-inequality}, respectively. The result follows from the combination of Theorem~\ref{T:disc.char.(3.3)} and Theorem~\ref{T:disc.char.(3.4)}.
\end{proof}

\section{Proofs}\label{S:Proofs}

\textbf{Proof of Theorem~\ref{T:main}}

\rm(i) Let $1 \leq p \leq \min\{r,q\}$. We have from [Theorem~\ref{T:disc.char.}, (i)] that $C \approx \mathcal{A}_1 + \mathcal{B}_1$. We will prove that $C_1 \approx \mathcal{A}_1 + \mathcal{B}_1$. First, we will show that $\mathcal{A}_1 + \mathcal{B}_1 \approx A_1 + \mathcal{B}_1$, where
\begin{equation*}
A_1 := \sup_{N+1 \leq  k} 2^{-\frac{k}{r}}  \esup_{t \in (a, x_{k})} \bigg(\int_t^{x_{k}} u \bigg)^{\frac{1}{q}} V_p(a, t).
\end{equation*}

It is clear that $\mathcal{A}_1 \leq A_1$. On the other hand, observe that 
\begin{align*}
A_1 & = \sup_{N+1 \leq  k} 2^{-\frac{k}{r}} \sup_{N+1\leq i \leq k} \esup_{t \in (x_{i-1}, x_{i})} \bigg(\int_t^{x_k} u \bigg)^{\frac{1}{q}} V_p(a, t)\\
& \approx \sup_{N+1 \leq  k} 2^{-\frac{k}{r}} \sup_{N+1\leq i \leq k} \esup_{t \in (x_{i-1}, x_{i})} \bigg(\int_t^{x_i} u \bigg)^{\frac{1}{q}} V_p(a, t) \\
& \quad + \sup_{N+2\leq  k} 2^{-\frac{k}{r}} \sup_{N+1\leq i < k} \bigg(\int_{x_i}^{x_k} u \bigg)^{\frac{1}{q}} V_p(a, x_i).
\end{align*}
Then, interchanging the supremum in the first term and applying  \eqref{3-sup-equiv} with $n=N+2$, for the second term, we have that
\begin{align*}
A_1 & \approx  \sup_{N+1 \leq  k} 2^{-\frac{k}{r}} \esup_{t \in (x_{k-1}, x_{k})} \bigg(\int_t^{x_k} u \bigg)^{\frac{1}{q}} V_p(a, t)  + \sup_{N+2\leq  k} 2^{-\frac{k}{r}} \bigg(\int_{x_{k-1}}^{x_k} u \bigg)^{\frac{1}{q}} V_p(a, x_{k-1}).
\end{align*}
Note that, for any $k \geq N+2$, we have
\begin{equation}\label{V-cut}
V_p(a,t) \approx V_p(a, x_{k-1}) + V_p(x_{k-1}, t), \quad \text{for every} \quad t\in (x_{k-1}, x_k).
\end{equation}
Then, in view of \eqref{V-cut}, 
\begin{align*}
A_1 & \approx  \sup_{N+1 \leq  k} 2^{-\frac{k}{r}} \esup_{t \in (x_{k-1}, x_{k})} \bigg(\int_t^{x_k} u \bigg)^{\frac{1}{q}} V_p(x_{k-1}, t)  + \sup_{N+2 \leq  k} 2^{-\frac{k}{r}} \bigg(\int_{x_{k-1}}^{x_k} u \bigg)^{\frac{1}{q}} V_p(a, x_{k-1})\\
& \lesssim \mathcal{A}_1 + \mathcal{B}_1.
\end{align*}
Then we have that $\mathcal{A}_1 + \mathcal{B}_1 \leq A_1 + \mathcal{B}_1 \lesssim \mathcal{A}_1 + \mathcal{B}_1$. 

It remains to show that $A_1 + \mathcal{B}_1 \approx C_1$. Applying \eqref{sup.equiv} with $\alpha = \frac{1}{r}$,
\begin{equation*}
A_1 \approx \esup_{x \in (a, b)} \bigg(\int_x^b w \bigg)^{\frac{1}{r}} \esup_{t\in (a,x)} \bigg(\int_t^x u \bigg)^{\frac{1}{q}} V_p(a, t)
\end{equation*}
holds, and interchanging supremum gives that
\begin{align}\label{A1<C1}
A_1 &\approx \esup_{t \in (a, b)} V_p(a, t) \esup_{x\in (t,b)} \bigg(\int_x^b w \bigg)^{\frac{1}{r}}  \bigg(\int_t^x u \bigg)^{\frac{1}{q}} \notag\\ 
& \leq \esup_{t \in (a, b)} V_p(a, t) \bigg(\int_t^b w(s) \bigg(\int_t^s u \bigg)^{\frac{r}{q}} ds \bigg)^{\frac{1}{r}}= C_1.
\end{align}
On the other hand, applying \eqref{int.equiv} with $\alpha= 0$, then using \eqref{dec-sum-sum} with $n=k+1$, we obtain for any $k\geq N$ that
\begin{align}\label{int-wu-sum}
\int_{x_k}^b w(s) \bigg(\int_{x_k}^s u \bigg)^{\frac{r}{q}} ds   \approx  \sum_{i=k+1}^{\infty} 2^{-i} \bigg(\int_{x_k}^{x_i} u \bigg)^{\frac{r}{q}} ds  \approx \sum_{i=k}^{\infty} 2^{-i} \bigg( \int_{x_i}^{x_{i+1}} u \bigg)^{\frac{r}{q}}.
\end{align}
Therefore, in view of \eqref{int-wu-sum},
\begin{align}
\mathcal{B}_1 & \approx  \sup_{N+1 \leq k} \bigg(\int_{x_k}^b w(s) \bigg(\int_{x_k}^s u \bigg)^{\frac{r}{q}} ds  \bigg)^{\frac{1}{r}} V_p(a, x_k)  \notag \\
& \leq \sup_{N+1 \leq k} \esup_{t \in (x_{k-1}\ x_k)} \bigg(\int_t^{b} w(s) \bigg(\int_t^{s} u \bigg)^{\frac{r}{q}} ds \bigg)^{\frac{1}{r}} V_p(a, t)  
= C_1. \label{B1*<C1}
\end{align}
Thus, combining \eqref{A1<C1} with \eqref{B1*<C1}, we have that $A_1 + \mathcal{B}_1 \lesssim C_1$. 

Conversely, using \eqref{V-cut}, we have
\begin{align*}
C_1 & \approx  \sup_{N+1 \leq k}   \bigg(\int_{x_k}^{b} w(s) \bigg(\int_{x_k}^s u \bigg)^{\frac{r}{q}} ds \bigg)^{\frac{1}{r}}  V_p(x_{k-1}, x_k) \\
&+\sup_{N+1 \leq k} 2^{-\frac{k}{r}}\esup_{t \in (x_{k-1}, x_k)}  \bigg(\int_{t}^{x_k} u \bigg)^{\frac{1}{q}}  V_p(x_{k-1}, t) \\
&+\sup_{N+1 \leq k} \esup_{t \in (x_{k-1}, x_k)}  \bigg(\int_{t}^{x_k} w(s) \bigg(\int_{t}^s u \bigg)^{\frac{r}{q}} ds \bigg)^{\frac{1}{r}}  V_p(x_{k-1}, t)\\
	& +  \sup_{N+2 \leq k} \bigg(\int_{x_{k-1}}^b w(s) \bigg(\int_{x_{k-1}}^s u \bigg)^{\frac{r}{q}} ds \bigg)^{\frac{1}{r}}  V_p(a, x_{k-1}).
\end{align*}
Then, using \eqref{int-wu-sum}, we arrive at
\begin{align} \label{C1<A1+B1*}
C_1 & \lesssim \sup_{N+1\leq k} 2^{-\frac{k}{r}} \esup_{t \in (x_{k-1}, x_k)} \bigg(\int_t^{x_k} u \bigg)^{\frac{1}{q}}  V_p(a, t) \notag\\
& \quad + \sup_{N+1 \leq k} \bigg(\sum_{i=k}^{\infty} 2^{-i} \bigg(\int_{x_i}^{x_{i+1}} u \bigg)^{\frac{r}{q}} \bigg)^{\frac{1}{r}}  V_p(a, x_{k}) \notag \\
	& \leq A_1 + \mathcal{B}_1.
\end{align}

As a result, we arrive at the conclusion that the best constant $C$ in \eqref{main-iterated} satisfies $C \approx C_1$.

\rm(ii) Let $r < p \leq q$. Then, we have from [Theorem~\ref{T:disc.char.}, (ii)] that the best constant in \eqref{main-iterated} satisfies $C\approx \mathcal{A}_2 + \mathcal{B}_2$. We will start by showing that $\mathcal{A}_2 + \mathcal{B}_2 \approx A_2 + B_2$, where
\begin{equation*}
A_2 := \bigg(\sum_{k=N+1}^{\infty} 2^{-k\frac{p}{p-r}} \esup_{t \in (a, x_{k})} \bigg(\int_t^{x_{k}} u\bigg)^{\frac{pr}{q(p-r)}} V_p(a, t)^{\frac{pr}{p-r}} \bigg)^{\frac{p-r}{pr}}
\end{equation*}
and
\begin{equation} \label{B2}
B_2 := \bigg(\sum_{k=N+1}^{\infty} 2^{-k} \bigg(\int_{x_k}^{x_{k+1}} u \bigg)^{\frac{r}{q}} \bigg(\sum_{i=k+2}^{\infty} 2^{-i} \bigg(\int_{x_i}^{x_{i+1}} u \bigg)^{\frac{r}{q}}  \bigg)^{\frac{r}{p-r}} 	V_p(a, x_k)^{\frac{pr}{p-r}} \bigg)^{\frac{p-r}{pr}}.
\end{equation}
We have $\mathcal{A}_2 \leq A_2$ by the definitions of $\mathcal{A}_2$ and $A_2$. On the other hand, 
\begin{align}
\mathcal{B}_2 &\approx B_2  + \bigg(\sum_{k=N+1}^{\infty} 2^{-k\frac{p}{p-r}} \bigg(\int_{x_k}^{x_{k+1}} u \bigg)^{\frac{rp}{q(p-r)}}  V_p(a, x_k)^{\frac{pr}{p-r}} \bigg)^{\frac{p-r}{pr}} \notag \\
& \hspace{2cm} + 
\bigg(\sum_{k=N+1}^{\infty} 2^{-k\frac{p}{p-r}} \bigg(\int_{x_{k}}^{x_{k+1}} u \bigg)^{\frac{r}{q}}  \bigg(\int_{x_{k+1}}^{x_{k+2}} u \bigg)^{\frac{r^2}{q(p-r)}}  V_p(a, x_k)^{\frac{r}{p-r}} \bigg)^{\frac{p-r}{pr}} \notag\\
& \lesssim B_2 + \bigg(\sum_{k=N+1}^{\infty} 2^{-k\frac{p}{p-r}} \bigg(\int_{x_k}^{x_{k+2}} u \bigg)^{\frac{pr}{q(p-r)}}  V_p(a, x_k)^{\frac{pr}{p-r}} \bigg)^{\frac{p-r}{pr}} \notag \\
& \lesssim B_2 + \bigg(\sum_{k=N+1}^{\infty} 2^{-k\frac{p}{p-r}} \esup_{t \in (a, x_k)} \bigg(\int_{t}^{x_{k+2}} u \bigg)^{\frac{pr}{q(p-r)}}  V_p(a, t)^{\frac{pr}{p-r}} \bigg)^{\frac{p-r}{pr}} \notag \\
& \lesssim B_2 + A_2 \label{B2*<A2+B2}
\end{align}
holds, hence $\mathcal{A}_2 + \mathcal{B}_2 \lesssim A_2 + B_2$. 

Next, we will show that $A_2 \lesssim \mathcal{A}_2 + \mathcal{B}_2$. Observe that,
\begin{align*}
A_2 & = \bigg(\sum_{k=N+1}^{\infty} 2^{-k\frac{p}{p-r}}  \sup_{N+1 \leq i \leq k} \esup_{t \in (x_{i-1}, x_{i})} \bigg(\int_t^{x_{k}} u\bigg)^{\frac{pr}{q(p-r)}} V_p(a, t)^{\frac{pr}{p-r}} \bigg)^{\frac{p-r}{pr}}  \\
& \approx \bigg(\sum_{k=N+1}^{\infty} 2^{-k\frac{p}{p-r}}  \sup_{N+1 \leq i \leq k} \esup_{t \in (x_{i-1}, x_{i})} \bigg(\int_t^{x_{i}} u\bigg)^{\frac{pr}{q(p-r)}} V_p(a, t)^{\frac{pr}{p-r}} \bigg)^{\frac{p-r}{pr}}  \\
& \quad + \bigg(\sum_{k=N+2}^{\infty} 2^{-k\frac{p}{p-r}}  \sup_{N+1 \leq i < k} \bigg(\int_{x_i}^{x_{k}} u\bigg)^{\frac{pr}{q(p-r)}} V_p(a, x_i)^{\frac{pr}{p-r}} \bigg)^{\frac{p-r}{pr}}.
\end{align*}
Applying \eqref{sum-sup} for the first term and \eqref{3-sum-equiv} for the second term, we obtain that
\begin{align*}
A_2 &  \approx \bigg(\sum_{k=N+1}^{\infty} 2^{-k\frac{p}{p-r}} \esup_{t \in (x_{k-1}, x_{k})} \bigg(\int_t^{x_{k}} u\bigg)^{\frac{pr}{q(p-r)}} V_p(a, t)^{\frac{pr}{p-r}} \bigg)^{\frac{p-r}{pr}}  \\
& \quad + \bigg(\sum_{k=N+2}^{\infty} 2^{-k\frac{p}{p-r}}  \bigg(\int_{x_{k-1}}^{x_{k}} u\bigg)^{\frac{pr}{q(p-r)}} V_p(a, x_{k-1})^{\frac{pr}{p-r}} \bigg)^{\frac{p-r}{pr}}.
\end{align*}
Using
\eqref{V-cut}
we arrive at
\begin{align*}
A_2 & \lesssim \mathcal{A}_2 + \bigg(\sum_{k=N+1}^{\infty} 2^{-k\frac{p}{p-r}} \bigg(\int_{x_{k}}^{x_{k+1}} u\bigg)^{\frac{pr}{q(p-r)}} V_p(a, x_{k})^{\frac{pr}{p-r}} \bigg)^{\frac{p-r}{pr}} \lesssim \mathcal{A}_2 + \mathcal{B}_2.
\end{align*}
Furthermore, it is clear from the definitions of $\mathcal{B}_2$ and $B_2$ that $B_2 \leq \mathcal{B}_2$. Then, we have $A_2 + B_2 \lesssim \mathcal{A}_2 + \mathcal{B}_2$, as well. Consequently, $C \approx A_2 + B_2$ holds.  

Next, we will prove that $A_2 + B_2 \approx C_2 + C_3$.
First of all, applying \eqref{int.equiv} with $\alpha = \frac{r}{p-r}$ and
\begin{equation*}
    h(x) = \esup_{t \in (a, x)} \bigg(\int_t^{x} u \bigg)^{\frac{pr}{q(p-r)}}  V_p(a, t)^{\frac{pr}{p-r}}, \quad x\in (a, b),
\end{equation*}
it is clear that
\begin{equation}\label{A2=C2}
A_2 \approx \bigg(\int_a^b \bigg(\int_x^b w \bigg)^{\frac{r}{p-r}} w(x) \esup_{t \in (a, x)} \bigg(\int_t^{x} u \bigg)^{\frac{pr}{q(p-r)}}  V_p(a, t)^{\frac{pr}{p-r}} dx \bigg)^{\frac{p-r}{pr}} = C_2
\end{equation}

On the other hand, using \eqref{int-wu-sum},
\begin{align}
B_2 & \approx \bigg(\sum_{k=N+1}^{\infty} 2^{-k}  \bigg(\int_{x_{k+2}}^b w(s) \bigg(\int_{x_{k+2}}^s u \bigg)^{\frac{r}{q}} ds  \bigg)^{\frac{r}{p-r}} \bigg(\int_{x_k}^{x_{k+1}} u \bigg)^{\frac{r}{q}}  V_p(a, x_k)^{\frac{pr}{p-r}} \bigg)^{\frac{p-r}{pr}}\notag\\
& \leq \bigg(\sum_{k=N+1}^{\infty} 2^{-k}  \bigg(\int_{x_{k+2}}^b w(s) \bigg(\int_{x_{k+2}}^s u \bigg)^{\frac{r}{q}} ds  \bigg)^{\frac{r}{p-r}} \notag\\
&\hspace{5cm} \times \esup_{t \in (a, x_{k+1})} \bigg(\int_{t}^{x_{k+1}} u \bigg)^{\frac{r}{q}}  V_p(a, t)^{\frac{r}{p-r}} \bigg)^{\frac{p-r}{pr}}\notag\\
& \approx\bigg(\sum_{k=N+1}^{\infty} \int_{x_{k+1}}^{x_{k+2}} w(x) dx  \bigg(\int_{x_{k+2}}^b w(s) \bigg(\int_{x_{k+2}}^s  u \bigg)^{\frac{r}{q}} ds  \bigg)^{\frac{r}{p-r}} \notag\\
& \hspace{5cm} \times \esup_{t \in (a, x_{k+1})} \bigg(\int_{t}^{x_{k+1}} u \bigg)^{\frac{r}{q}}  V_p(a, t)^{\frac{pr}{p-r}} \bigg)^{\frac{p-r}{pr}}\notag\\
& \leq \bigg(\sum_{k=N+1}^{\infty} \int_{x_{k+1}}^{x_{k+2}} w(x)  \bigg(\int_{x}^b w(s) \bigg(\int_{x}^s u \bigg)^{\frac{r}{q}} ds  \bigg)^{\frac{r}{p-r}} \notag\\
&\hspace{5cm} \times \esup_{t \in (a, x)} \bigg(\int_{t}^{x} u \bigg)^{\frac{r}{q}}  V_p(a, t)^{\frac{pr}{p-r}} dx \bigg)^{\frac{p-r}{pr}} \notag\\
&\leq C_3. \label{B2<C3}
\end{align}
Combination of \eqref{A2=C2} and \eqref{B2<C3} yield that $A_2 + B_2 \lesssim C_2 + C_3$. 

Conversely,
\begin{align*}
C_3 & = \bigg( \sum_{k=N+1}^{\infty} \int_{x_{k-1}}^{x_k}  \bigg(\int_{x}^b w(s) \bigg(\int_{x}^s u \bigg)^{\frac{r}{q}} ds  \bigg)^{\frac{r}{p-r}} w(x) \esup_{t \in (a, x)} \bigg(\int_{t}^{x} u \bigg)^{\frac{r}{q}} V_p(a, t)^{\frac{pr}{p-r}}  dx \bigg)^{\frac{p-r}{pr}} \\
& \approx \bigg( \sum_{k=N+1}^{\infty} \int_{x_{k-1}}^{x_k}  \bigg(\int_{x}^{x_k} w(s) \bigg(\int_{x}^s u \bigg)^{\frac{r}{q}} ds  \bigg)^{\frac{r}{p-r}} w(x) \esup_{t \in (a, x)} \bigg(\int_{t}^{x} u \bigg)^{\frac{r}{q}} V_p(a, t)^{\frac{pr}{p-r}} dx \bigg)^{\frac{p-r}{pr}} \\
& \hspace{0.2cm} + \bigg( \sum_{k=N+1}^{\infty} \int_{x_{k-1}}^{x_k}  \bigg(\int_{x_k}^b w(s) \bigg(\int_{x}^s u \bigg)^{\frac{r}{q}} ds  \bigg)^{\frac{r}{p-r}} w(x) \esup_{t \in (a, x)} \bigg(\int_{t}^{x} u \bigg)^{\frac{r}{q}} V_p(a, t)^{\frac{pr}{p-r}}  dx \bigg)^{\frac{p-r}{pr}} \\
& \approx \bigg( \sum_{k=N+1}^{\infty} \int_{x_{k-1}}^{x_k}  \bigg(\int_{x}^{x_k} w(s) \bigg(\int_{x}^s u \bigg)^{\frac{r}{q}} ds  \bigg)^{\frac{r}{p-r}} w(x) \esup_{t \in (a, x)} \bigg(\int_{t}^{x} u \bigg)^{\frac{r}{q}}  V_p(a, t)^{\frac{pr}{p-r}}  dx \bigg)^{\frac{p-r}{pr}} \\
& \hspace{0.2cm} + \bigg( \sum_{k=N+1}^{\infty} 2^{-k\frac{r}{p-r}} \int_{x_{k-1}}^{x_k} \bigg(\int_{x}^{x_k} u \bigg)^{\frac{r^2}{q(p-r)}}   w(x) \esup_{t \in (a, x)} \bigg(\int_{t}^{x} u \bigg)^{\frac{r}{q}}  V_p(a, t)^{\frac{pr}{p-r}}  dx \bigg)^{\frac{p-r}{pr}} \\
& \hspace{0.2cm}  + \bigg( \sum_{k=N+1}^{\infty} \bigg(\int_{x_k}^b w(s) \bigg(\int_{x_k}^s u \bigg)^{\frac{r}{q}} ds  \bigg)^{\frac{r}{p-r}} \int_{x_{k-1}}^{x_k} w(x) \esup_{t \in (a, x)} \bigg(\int_{t}^{x} u \bigg)^{\frac{r}{q}}  V_p(a, t)^{\frac{pr}{p-r}}  dx \bigg)^{\frac{p-r}{pr}} \\
& =: C_{3,1} + C_{3,2} + C_{3,3}.
\end{align*}
It is easy to see that
\begin{align*}
C_{3,1} & \leq \bigg( \sum_{k=N+1}^{\infty} \int_{x_{k-1}}^{x_k}  \bigg(\int_{x}^{x_k} w \bigg)^{\frac{r}{p-r}} w(x) \esup_{t \in (a, x)} \bigg(\int_{t}^{x_k} u \bigg)^{\frac{pr}{q(p-r)}}  V_p(a, t)^{\frac{pr}{p-r}} dx \bigg)^{\frac{p-r}{pr}} \\
& \leq \bigg( \sum_{k=N+1}^{\infty}  \int_{x_{k-1}}^{x_k}  \bigg(\int_{x}^{x_k} w \bigg)^{\frac{r}{p-r}} w(x) dx \, \esup_{t \in (a, x_k)} \bigg(\int_{t}^{x_k} u \bigg)^{\frac{pr}{q(p-r)}}  V_p(a, t)^{\frac{pr}{p-r}} \bigg)^{\frac{p-r}{pr}} \\
& \approx A_2,
\end{align*}
and
\begin{align*}
C_{3,2} &\leq \bigg( \sum_{k=N+1}^{\infty} 2^{-k\frac{r}{p-r}} \int_{x_{k-1}}^{x_k} w(x)  \esup_{t \in (a, x)} \bigg(\int_{t}^{x_k} u \bigg)^{\frac{pr}{q(p-r)}}  V_p(a, t)^{\frac{pr}{p-r}} dx \bigg)^{\frac{p-r}{pr}} \leq A_2
\end{align*}
hold. Furthermore,
\begin{align*}
C_{3,3} & \lesssim \bigg( \sum_{k=N+1}^{\infty} 2^{-k} \bigg(\int_{x_k}^b w(s) \bigg(\int_{x_k}^s u \bigg)^{\frac{r}{q}} ds  \bigg)^{\frac{r}{p-r}}  \esup_{t \in (a, x_k)} \bigg(\int_{t}^{x_k} u \bigg)^{\frac{r}{q}}  V_p(a, t)^{\frac{pr}{p-r}}  \bigg)^{\frac{p-r}{pr}}\\
& = \bigg( \sum_{k=N+1}^{\infty} 2^{-k} \bigg(\int_{x_k}^b w(s) \bigg(\int_{x_k}^s u \bigg)^{\frac{r}{q}} ds  \bigg)^{\frac{r}{p-r}} \\
& \hspace{4cm} \times  \sup_{N+1 \leq i \leq  k} \esup_{t \in (x_{i-1}, x_i)} \bigg(\int_{t}^{x_k} u \bigg)^{\frac{r}{q}}  V_p(a, t)^{\frac{pr}{p-r}}  \bigg)^{\frac{p-r}{pr}}\\
& \approx \bigg( \sum_{k=N+1}^{\infty} 2^{-k} \bigg(\int_{x_k}^b w(s) \bigg(\int_{x_k}^s u \bigg)^{\frac{r}{q}} ds  \bigg)^{\frac{r}{p-r}} \\
& \hspace{4cm} \times \sup_{N+1 \leq i \leq k} \esup_{t \in (x_{i-1}, x_i)} \bigg(\int_{t}^{x_i} u \bigg)^{\frac{r}{q}} V_p(a, t)^{\frac{pr}{p-r}}  \bigg)^{\frac{p-r}{pr}} \\
& + \bigg( \sum_{k=N+2}^{\infty} 2^{-k} \bigg(\int_{x_k}^b w(s) \bigg(\int_{x_k}^s u \bigg)^{\frac{r}{q}} ds  \bigg)^{\frac{r}{p-r}}  \sup_{N+1 \leq i <  k} \bigg(\int_{x_i}^{x_{k}} u \bigg)^{\frac{r}{q}} V_p(a, x_i)^{\frac{pr}{p-r}} \bigg)^{\frac{p-r}{pr}}  \\
& =: I + II.
\end{align*}
Since, the sequence $\{a_k\}_{k=N+1}^{\infty}$, with 
$$
a_k =: 2^{-k} \bigg(\int_{x_k}^b w(s) \bigg(\int_{x_k}^s u \bigg)^{\frac{r}{q}} ds  \bigg)^{\frac{r}{p-r}} 
$$
is geometrically decreasing, \eqref{sum-sup} yields that
\begin{align*}
I & \approx \bigg( \sum_{k=N+1}^{\infty} 2^{-k}  \bigg(\int_{x_k}^b w(s) \bigg(\int_{x_k}^s u \bigg)^{\frac{r}{q}} ds  \bigg)^{\frac{r}{p-r}} \esup_{t \in (x_{k-1}, x_k)} \bigg(\int_{t}^{x_k} u \bigg)^{\frac{r}{q}}  V_p(a, t)^{\frac{pr}{p-r}}  \bigg)^{\frac{p-r}{pr}}.
\end{align*}
Let  $y_k \in [x_{k-1}, x_k]$, $N \leq k$ be such that
\begin{align}
\esup_{t \in (x_{k-1}, x_k)} \bigg(\int_{t}^{x_k} u \bigg)^{\frac{r}{q}}  V_p(a, t)^{\frac{pr}{p-r}}   \lesssim \bigg(\int_{y_k}^{x_k} u \bigg)^{\frac{r}{q}}  V_p(a, y_k)^{\frac{pr}{p-r}}. 
\end{align}
Then, we have
\begin{align*}
I & \lesssim \bigg( \sum_{k=N+1}^{\infty} 2^{-k} \bigg(\int_{x_k}^b w(s) \bigg(\int_{x_k}^s u \bigg)^{\frac{r}{q}} ds  \bigg)^{\frac{r}{p-r}} \bigg(\int_{y_k}^{x_k} u \bigg)^{\frac{r}{q}}  V_p(a, y_k)^{\frac{pr}{p-r}} \bigg)^{\frac{p-r}{pr}}
\end{align*}
Thus, \eqref{int-wu-sum} ensures that
\begin{align*}
I & \lesssim \bigg( \sum_{k=N+1}^{\infty} 2^{-k} \bigg(\sum_{i=k}^{\infty} 2^{-i} \bigg( \int_{x_i}^{x_{i+1}} u \bigg)^{\frac{r}{q}}  \bigg)^{\frac{r}{p-r}} \bigg(\int_{y_k}^{x_k} u \bigg)^{\frac{r}{q}}  V_p(a, y_k)^{\frac{pr}{p-r}} \bigg)^{\frac{p-r}{pr}}\\
& \leq
\bigg( \sum_{k=N+1}^{\infty} 2^{-k} \bigg(\sum_{i=k}^{\infty} 2^{-i} \bigg( \int_{y_i}^{y_{i+2}} u \bigg)^{\frac{r}{q}}  \bigg)^{\frac{r}{p-r}} \bigg(\int_{y_k}^{y_{k+1}} u \bigg)^{\frac{r}{q}}  V_p(a, y_k)^{\frac{pr}{p-r}} \bigg)^{\frac{p-r}{pr}}.
\end{align*}
Moreover,
\begin{equation*}
2^{-k} \approx \int_{x_k}^b w \leq \int_{y_k}^b w \leq \int_{x_{k-1}}^b w \approx 2^{-(k-1)}, \quad N+1 \leq k. 
\end{equation*}
As a result, $\{y_k\}_{k=N+1}^{\infty}$ is a discretizing sequence of $W^*$, as well. This fact together with \eqref{B2*<A2+B2} yield $I \lesssim \mathcal{B}_2 \lesssim A_2 +B_2$. 

On the other hand, applying \eqref{3-sum-equiv} with
$$
\tau_k =  2^{-k} \bigg(\int_{x_k}^b w(s) \bigg(\int_{x_k}^s u \bigg)^{\frac{r}{q}} ds  \bigg)^{\frac{r}{p-r}}, \quad \sigma_k = V_p(a, x_k)^{\frac{pr}{p-r}}, \quad \alpha=\frac{r}{q},
$$
gives
\begin{align*}
II \approx  \bigg( \sum_{k=N+2}^{\infty} 2^{-k} \bigg(\int_{x_k}^b w(s) \bigg(\int_{x_k}^s u \bigg)^{\frac{r}{q}} ds  \bigg)^{\frac{r}{p-r}}  \bigg(\int_{x_{k-1}}^{x_{k}} u \bigg)^{\frac{r}{q}} V_p(a, x_{k-1})^{\frac{pr}{p-r}} \bigg)^{\frac{p-r}{pr}}.  
\end{align*}
Finally, using \eqref{int-wu-sum} and \eqref{B2*<A2+B2}, we obtain that
\begin{align*}
II &  \approx  \bigg( \sum_{k=N+2}^{\infty} 2^{-k} \bigg(\sum_{i=k}^{\infty} 2^{-i} \bigg( \int_{x_i}^{x_{i+1}} u \bigg)^{\frac{r}{q}} \bigg)^{\frac{r}{1-r}} \bigg(\int_{x_{k-1}}^{x_k} u \bigg)^{\frac{r}{q}} V_p(a, x_{k-1})^{\frac{pr}{p-r}}  \bigg)^{\frac{p-r}{pr}}
\leq A_2 + B_2.
\end{align*}
Therefore, 
\begin{equation}\label{C3<A2+B2}
C_3 \lesssim A_2 + B_2.
\end{equation}

Finally, combination of \eqref{A2=C2} and \eqref{C3<A2+B2} yield, $C_2 + C_3 \approx  A_2 + B_2$. Accordingly, the best constant $C$ in \eqref{main-iterated} satisfies $C \approx C_2 +C_3 $.

\rm(iii) Let $ q< p \le r$. According to [Theorem~\ref{T:disc.char.}, (iii)], the best constant in \eqref{main-iterated} satisfies $C\approx \mathcal{A}_3 + \mathcal{B}_1$. We will begin our proof by showing that $\mathcal{A}_3 + \mathcal{B}_1 \approx A_3 + \mathcal{B}_1$, where
\begin{equation*}
A_3 := \sup_{N+1 \leq k <\infty} 2^{-\frac{k}{r}} \bigg(\int_a^{x_{k}} \bigg(\int_t^{x_{k}} u \bigg)^{\frac{q}{p-q}} u(t) \,  V_p(a,t)^{\frac{pq}{p-q}}dt \bigg)^{\frac{p-q}{pq}}.
\end{equation*}
It is clear that $\mathcal{A}_3 \leq A_3$, the proof of this part is complete if we show that $A_3 \lesssim \mathcal{A}_3 + \mathcal{B}_1$. Assume that $\max\{\mathcal{A}_3, \mathcal{B}_1\}<\infty$. Then,
\begin{equation*}
\bigg(\int_{x_{k-1}}^{x_{k}} \bigg(\int_t^{x_{k}} u \bigg)^{\frac{q}{p-q}} u(t) \, V_p(a,t)^{\frac{pq}{p-q}}dt \bigg)^{\frac{p-q}{pq}} <\infty, \quad k\geq N+1
\end{equation*}
holds. Thus, for each $t\in [x_{k-1}, x_k]$, $k\geq N+1$, we have
\begin{equation*}
\bigg(\int_t^{x_{k}} u \bigg)^{\frac{1}{q}} V_p(a,t) \lesssim \bigg(\int_{t}^{x_{k}} \bigg(\int_s^{x_{k}} u \bigg)^{\frac{q}{p-q}} u(s) \, V_p(a,s)^{\frac{pq}{p-q}} ds \bigg)^{\frac{p-q}{pq}}.
\end{equation*}
Therefore, we have 
\begin{equation}\label{lim-0}
\lim_{t \rightarrow x_k-}    \bigg(\int_t^{x_{k}} u \bigg)^{\frac{1}{q}} V_p(a,t) = 0. 
\end{equation}
In that case, integration by parts gives
\begin{align}
\bigg(\int_a^{x_{k}} & \bigg(\int_t^{x_{k}} u \bigg)^{\frac{q}{p-q}} u(t) \, V_p(a,t)^{\frac{pq}{p-q}}dt \bigg)^{\frac{p-q}{pq}} \notag \\
& \approx \bigg(\int_a^{x_{k}} \bigg(\int_t^{x_{k}} u \bigg)^{\frac{p}{p-q}}  d\big[ V_p(a,t)^{\frac{pq}{p-q}} \big] \bigg)^{\frac{p-q}{pq}} + \lim_{t \rightarrow a+} \bigg(\int_t^{x_{k}} u \bigg)^{\frac{1}{q}}  V_p(a,t)\notag \\
& \approx \bigg(\sum_{i=N+1}^k \int_{x_{i-1}}^{x_{i}} \bigg(\int_t^{x_{i}} u \bigg)^{\frac{p}{p-q}}  d\big[V_p(a,t)^{\frac{pq}{p-q}}\big]   \bigg)^{\frac{p-q}{pq}} \notag\\
& \quad +  \bigg(\sum_{i=N+1}^{k-1} \bigg(\int_{x_{i}}^{x_k} u \bigg)^{\frac{p}{p-q}} \int_{x_{i-1}}^{x_{i}} d\big[V_p(a,t)^{\frac{pq}{p-q}}\big]   \bigg)^{\frac{p-q}{pq}} \notag
\\
& \qquad +  \lim_{t \rightarrow a+} \bigg(\int_t^{x_{k}} u \bigg)^{\frac{1}{q}}  V_p(a,t).\notag
\end{align}
Moreover,  Minkowski's inequality with $\frac{p}{p-q} > 1$ yields that 
\begin{align*}
\bigg(\sum_{i=N+1}^{k-1} & \bigg(\int_{x_{i}}^{x_k} u \bigg)^{\frac{p}{p-q}} \int_{x_{i-1}}^{x_{i}} d\big[V_p(a,t)^{\frac{pq}{p-q}}\big]   \bigg)^{\frac{p-q}{pq}} \\
& = \bigg(\sum_{i=N+1}^{k-1} \bigg(\sum_{j=i}^{k-1} \int_{x_{j}}^{x_{j+1}} u \bigg)^{\frac{p}{p-q}} \int_{x_{i-1}}^{x_{i}} d\big[V_p(a,t)^{\frac{pq}{p-q}}\big]   \bigg)^{\frac{p-q}{pq}}\\
& \leq  \bigg(\sum_{j=N+1}^{k-1} \bigg(\int_{x_{j}}^{x_{j+1}} u \bigg) \bigg(\sum_{i=N+1}^{j} \int_{x_{i-1}}^{x_i}  d\big[V_p(a,t)^{\frac{pq}{p-q}}\big]   \bigg)^{\frac{p-q}{p}} \bigg)^{\frac{1}{q}}\\
& =  \bigg(\sum_{j=N+1}^{k-1} \bigg(\int_{x_{j}}^{x_{j+1}} u \bigg) \bigg( \int_{a}^{x_{j}}  d\big[V_p(a,t)^{\frac{pq}{p-q}}\big]   \bigg)^{\frac{p-q}{p}} \bigg)^{\frac{1}{q}}.
\end{align*}
Then, we arrive at
\begin{align}\label{IBP-estimate}
\bigg(\int_a^{x_{k}} & \bigg(\int_t^{x_{k}} u \bigg)^{\frac{q}{p-q}} u(t) \, V_p(a,t)^{\frac{pq}{p-q}}dt \bigg)^{\frac{p-q}{pq}} \notag \\
& \lesssim \bigg(\sum_{i=N+1}^k \int_{x_{i-1}}^{x_{i}} \bigg(\int_t^{x_{i}} u \bigg)^{\frac{p}{p-q}}  d\big[V_p(a,t)^{\frac{pq}{p-q}}\big]   \bigg)^{\frac{p-q}{pq}} \notag\\
& \quad +  \bigg(\sum_{j=N+1}^{k-1} \bigg(\int_{x_{j}}^{x_{j+1}} u \bigg) \bigg(\int_{a}^{x_j}  d\big[V_p(a,t)^{\frac{pq}{p-q}}\big]   \bigg)^{\frac{p-q}{p}} \bigg)^{\frac{1}{q}} \notag\\
& \qquad +  \lim_{t \rightarrow a+} \bigg(\int_t^{x_{k}} u \bigg)^{\frac{1}{q}}  V_p(a,t).
\end{align}
Now, we are in position to find the upper estimate for $A_3$. Using  \eqref{IBP-estimate}, we have that
\begin{align*}
A_3 & \lesssim \sup_{N+1 \leq k <\infty } 2^{-\frac{k}{r}} \bigg(\sum_{i=N+1}^k \int_{x_{i-1}}^{x_{i}} \bigg(\int_t^{x_{i}} u \bigg)^{\frac{p}{p-q}}  d\big[V_p(a,t)^{\frac{pq}{p-q}}\big]   \bigg)^{\frac{p-q}{pq}}\\
& \quad + \sup_{N+2 \leq k <\infty} 2^{-\frac{k}{r}}  \bigg(\sum_{j=N+1}^{k-1} \bigg(\int_{x_{j}}^{x_{j+1}} u \bigg) \bigg(\int_{a}^{x_j}  d\big[V_p(a,t)^{\frac{pq}{p-q}}\big]   \bigg)^{\frac{p-q}{p}} \bigg)^{\frac{1}{q}} \notag\\
& \qquad +  \sup_{N+1 \leq k <\infty} 2^{-\frac{k}{r}} \lim_{t \rightarrow a+} \bigg(\int_t^{x_{k}} u \bigg)^{\frac{1}{q}}  V_p(a,t) \end{align*}
Further, \eqref{sup-sum} yields, 
\begin{align*}
A_3 & \lesssim 
\sup_{N+1 \leq k <\infty} 2^{-\frac{k}{r}} \bigg( \int_{x_{k-1}}^{x_{k}} \bigg(\int_t^{x_{k}} u \bigg)^{\frac{p}{p-q}}  d\big[V_p(a,t)^{\frac{pq}{p-q}}\big]   \bigg)^{\frac{p-q}{pq}} \notag\\
& \quad + \sup_{N+1 \leq k <\infty} 2^{-\frac{k}{r}} \bigg(\int_{x_{k}}^{x_{k+1}} u \bigg)^{\frac{1}{q}} \bigg( \int_{a}^{x_{k}}  d\big[V_p(a,t)^{\frac{pq}{p-q}}\big]\bigg)^{\frac{p-q}{pq}} \notag\\
& \qquad +  \sup_{N+1 \leq k <\infty } 2^{-\frac{k}{r}} \lim_{t \rightarrow a+} \bigg(\int_t^{x_{k}} u \bigg)^{\frac{1}{q}}  V_p(a,t)\\
& =: A_{3,1} + A_{3,2} + A_{3,3}.
\end{align*}
Integrating by parts again, we have that 
\begin{align*}
A_{3,1} & \lesssim \sup_{N+1 \leq k <\infty} 2^{-\frac{k}{r}} \bigg(\int_{x_{k-1}}^{x_{k}} \bigg(\int_t^{x_{k}} u \bigg)^{\frac{q}{p-q}} u(t)  V_p(a,t)^{\frac{pq}{p-q}} dt    \bigg)^{\frac{p-q}{pq}}. 
\end{align*}
Thus, \eqref{V-cut} gives,
\begin{align}\label{A_31<A3*+B1*}
A_{3,1} & \lesssim \sup_{N+1 \leq k <\infty} 2^{-\frac{k}{r}} \bigg(\int_{x_{k-1}}^{x_{k}} \bigg(\int_t^{x_{k}} u \bigg)^{\frac{q}{p-q}} u(t)  V_p(x_{k-1},t)^{\frac{pq}{p-q}} dt  \bigg)^{\frac{p-q}{pq}} \notag\\
& \quad + \sup_{N+2 \leq k <\infty} 2^{-\frac{k}{r}} \bigg(\int_{x_{k-1}}^{x_{k}} u \bigg)^{\frac{1}{q}}  V_p(a,x_{k-1}) \notag\\
& \lesssim \mathcal{A}_3 + \mathcal{B}_1. 
\end{align}
Additionally,
\begin{align}\label{A_32<A3*+B1*}
A_{3,2} \lesssim \sup_{N+1 \leq k <\infty} 2^{-\frac{k}{r}} \bigg(\int_{x_{k}}^{x_{k+1}} u \bigg)^{\frac{1}{q}} V_p(a,x_k)  \leq \mathcal{B}_1.
\end{align}

Lastly, we will find a suitable upper estimate for $A_{3,3}$. To this end, we will treat the cases $N= -\infty$ and $N <\infty$, separately. Observe that, if $N=-\infty$, since $x_i \rightarrow a$ if $i\rightarrow -\infty$, we have for any $k$,
\begin{align}\label{lim<sup-1}
\lim_{t \rightarrow a+} \bigg(\int_t^{x_{k}} u \bigg)^{\frac{1}{q}}  V_p(a,t) = \lim_{i \rightarrow -\infty} \bigg(\int_{x_i}^{x_{k}} u \bigg)^{\frac{1}{q}}  V_p(a,x_i) \leq \sup_{i < k}  \bigg(\int_{x_i}^{x_{k}} u \bigg)^{\frac{1}{q}} V_p(a,x_i).    
\end{align}
Then, then using \eqref{lim<sup-1} together with \eqref{3-sup-equiv}, we get
\begin{align*}
A_{3,3} & \lesssim  \sup_{k\in \mathbb{Z} } 2^{-\frac{k}{r}} \sup_{i < k}  \bigg(\int_{x_i}^{x_{k}} u \bigg)^{\frac{1}{q}}    V_p(a,x_i) \\
&\approx \sup_{k \in \mathbb{Z}} 2^{-\frac{k}{r}} \bigg(\int_{x_{k-1}}^{x_{k}} u \bigg)^{\frac{1}{q}} V_p(a,x_{k-1}) \lesssim \mathcal{B}_1. 
\end{align*}
On the other hand, if, $N > -\infty$
\begin{align} \label{lim<sup-2}
 \lim_{t \rightarrow a+} \bigg(\int_t^{x_{k}} u \bigg)^{\frac{1}{q}}  V_p(a,t) &\leq \esup_{t \in (a, x_{N+1})} \bigg(\int_t^{x_k} u \bigg)^{\frac{1}{q}}  V_p(a,t) \notag\\
 &\approx \esup_{t \in (a, x_{N+1})} \bigg(\int_t^{x_{N+1}} u \bigg)^{\frac{1}{q}}  V_p(a,t) +  \bigg(\int_{x_{N+1}}^{x_{k}} u \bigg)^{\frac{1}{q}}  V_p(a,x_{N+1}).
\end{align}
Additionally, it is easy to see that
\begin{align}
\esup_{\tau \in (x, y)} \bigg(\int_{\tau}^y u \bigg)^{\frac{1}{q}} V_p(x,\tau) \leq \bigg(\int_x^y \bigg(\int_t^y u \bigg)^{\frac{q}{p-q}} u(t) V_p(x,t)^{\frac{pq}{p-q}} dt  \bigg)^{\frac{p-q}{pq}}. \label{sup-int-estimate}
\end{align}
First, using \eqref{lim<sup-2}, we get
\begin{align*}
A_{3,3} & \lesssim  \sup_{N+1 \leq k <\infty } 2^{-\frac{k}{r}} \esup_{t \in (a, x_{N+1})} \bigg(\int_t^{x_{N+1}} u \bigg)^{\frac{1}{q}}  V_p(a,t) \\
& \hspace{2cm} + \sup_{N+2 \leq k <\infty } 2^{-\frac{k}{r}} \sup_{N+1 \leq i < k }\bigg(\int_{x_{i}}^{x_{k}} u \bigg)^{\frac{1}{q}}  V_p(a,x_{i}).
\end{align*}
Then, using \eqref{sup-int-estimate} for the first term and applying \eqref{3-sup-equiv} for the second term, we get
\begin{align*}
A_{3,3} & \lesssim  \sup_{N+1 \leq k <\infty } 2^{-\frac{k}{r}} \bigg(\int_{x_{k-1}}^{x_k} \bigg(\int_t^{x_k} u \bigg)^{\frac{q}{p-q}} u(t) \,  V_p(a,t)^{\frac{pq}{p-q}}dt \bigg)^{\frac{p-q}{pq}} \\
& \hspace{2cm} + \sup_{N+2 \leq k <\infty } 2^{-\frac{k}{r}} \bigg(\int_{x_{k-1}}^{x_k} u \bigg)^{\frac{1}{q}} V_p(a, x_{k-1}).
\end{align*}
Finally, using \eqref{V-cut}, we arrive at
\begin{align*}
A_{3,3}  \lesssim  \mathcal{A}_3 + \sup_{N+2 \leq k <\infty } 2^{-\frac{k}{r}} \bigg(\int_{x_{k-1}}^{x_k} u \bigg)^{\frac{1}{q}} V_p(a, x_{k-1}) \lesssim \mathcal{A}_3 + \mathcal{B}_1.
\end{align*}
Consequently, we have for any $N \in \mathbb{Z}\cup \{-\infty\}$, $A_{3,3} \lesssim \mathcal{A}_3 + \mathcal{B}_1$. Combining the last estimate with \eqref{A_31<A3*+B1*} and \eqref{A_32<A3*+B1*}, we arrive at $A_3 \lesssim \mathcal{A}_3 + \mathcal{B}_1$, hence, $C \approx A_3 + \mathcal{B}_1$.

Let us now continue the proof by showing $A_3 + \mathcal{B}_1\approx C_1+ C_4$. 

To this end, taking $\alpha = 1/r $ and 
\begin{equation*}
h(x) = \bigg(\int_a^{x} \bigg(\int_t^x u \bigg)^{\frac{q}{p-q}} u(t) V_p(a,t)^{\frac{pq}{p-q}}dt \bigg)^{\frac{p-q}{pq}}, \quad x \in (a, b)
\end{equation*}
in \eqref{sup.equiv}, we have that $A_3 \approx C_4$. Moreover, we have already shown in \eqref{B1*<C1} and \eqref{C1<A1+B1*} that $\mathcal{B}_1 \lesssim C_1 \lesssim A_1 + \mathcal{B}_1$. Furthermore, \eqref{sup-int-estimate} yields that $A_1 \lesssim A_3$. Consequently, we have $A_3 + \mathcal{B}_1 \lesssim C_4 + C_1 \lesssim A_3 + \mathcal{B}_1$, which is the desired estimate. 

\rm(iv) Let $\max\{r,q \}< p$. Then, using [Theorem~\ref{T:disc.char.}, (iv)], we have that $C \approx \mathcal{B}_2 + \mathcal{A}_4$. First of all, We will show that $\mathcal{B}_2 + \mathcal{A}_4 \approx B_2 + A_4$, where  
\begin{equation*}
A_4 := \bigg( \sum_{k=N+1}^{\infty} 2^{-k\frac{r}{p-r}} \bigg( \int_{a}^{x_{k}} \bigg(\int_{t}^{x_{k}} u \bigg)^{\frac{q}{p-q}} u(t) V_p(a, t)^{\frac{pq}{p-q}} dt  \bigg)^{\frac{r(p-q)}{q(p-r)}}  \bigg)^{\frac{p-r}{pr}},
\end{equation*} 
and $B_2$ is defined in \eqref{B2}.

It is clear that $\mathcal{A}_4 \leq A_4$. We have already shown in \eqref{B2*<A2+B2} that $\mathcal{B}_2 \lesssim A_2 + B_2$.  Moreover, analogously as in the previous proof, using \eqref{sup-int-estimate}, one can easily see that $A_2 \lesssim A_4$. Thus, $\mathcal{B}_2 + \mathcal{A}_4 \lesssim B_2 + A_4$ follows.  

It remains to prove that $B_2 + A_4 \lesssim \mathcal{B}_2 + \mathcal{A}_4$. Assume that $\max\{\mathcal{A}_4, \mathcal{B}_2\}<\infty$. Then, using the same steps as in the previous case, we can see that \eqref{lim-0} holds, therefore \eqref{IBP-estimate} is true in this case, as well. 

Applying \eqref{IBP-estimate} combined with \eqref{dec-sum-sum}, we obtain that
\begin{align*}
A_4 & \lesssim  \bigg( \sum_{k=N+1}^{\infty} 2^{-k\frac{r}{p-r}} \bigg( \int_{x_{k-1}}^{x_{k}} \bigg(\int_{t}^{x_{k}} u \bigg)^{\frac{p}{p-q}} d\bigg[ V_p(a, t)^{\frac{pq}{p-q}} \bigg]   \bigg)^{\frac{r(p-q)}{q(p-r)}}  \bigg)^{\frac{p-r}{pr}} \\
& \quad +  \bigg( \sum_{k=N+2}^{\infty} 2^{-k\frac{p}{p-r}} \bigg(\int_{x_{k-1}}^{x_{k}} u \bigg)^{\frac{pr}{q(p-r)}} \bigg( \int_{a}^{x_{k-1}} d\bigg[ V_p(a, t)^{\frac{pq}{p-q}} \bigg]  \bigg)^{\frac{r(p-q)}{q(p-r)}}  \bigg)^{\frac{p-r}{pr}}\\
& \qquad  + \bigg( \sum_{k=N+1}^{\infty} 2^{-k\frac{p}{p-r}}  \bigg[\lim_{t \rightarrow a+} \bigg(\int_{t}^{x_{k}} u \bigg)^{\frac{1}{q}} V_p(a, t)\bigg]^{\frac{pr}{p-r}} \bigg)^{\frac{p-r}{pr}} \\
& =: A_{4,1} + A_{4,2} + A_{4,3}.
\end{align*}
holds.

As in the proof of the previous case, using integration by parts in combination with \eqref{V-cut}, we have that
\begin{align}
A_{4,1} &   \lesssim \bigg( \sum_{k=N+1}^{\infty} 2^{-k\frac{p}{p-r}} \bigg( \int_{x_{k-1}}^{x_k} \bigg(\int_{t}^{x_k} u \bigg)^{\frac{q}{p-q}} u(t) V_p(a,t)^{\frac{pq}{p-q}} dt  \bigg)^{\frac{r(p-q)}{q(p-r)}}  \bigg)^{\frac{p-r}{pr}} \notag\\
& \approx \mathcal{A}_4 + \bigg( \sum_{k=N+2}^{\infty} 2^{-k\frac{p}{p-r}} \bigg( \int_{x_{k-1}}^{x_k}  u  \bigg)^{\frac{pr}{q(p-r)}} V_p(a, x_{k-1})^{\frac{pr}{p-r}}   \bigg)^{\frac{p-r}{pr}} \notag \\
& \lesssim \mathcal{A}_4 + \mathcal{B}_2.\label{A_41<A4+B2}
\end{align}
On the other hand, it is clear that
\begin{align}
A_{4,2} \lesssim \bigg( \sum_{k=N+2}^{\infty} 2^{-k\frac{p}{p-r}} \bigg( \int_{x_{k-1}}^{x_k}  u  \bigg)^{\frac{pr}{q(p-r)}} V_p(a, x_{k-1})^{\frac{pr}{p-r}}   \bigg)^{\frac{p-r}{pr}}
\lesssim \mathcal{B}_2. \label{A42<A4+B2}
\end{align}
Furthermore, if $N = -\infty$, using \eqref{lim<sup-1}, and then applying \eqref{3-sum-equiv}, we get 
\begin{align*}
A_{4,3} & \lesssim  \bigg( \sum_{k=-\infty}^{\infty} 2^{-k\frac{p}{p-r}} \sup_{-\infty < i < k}  \bigg(\int_{x_{i}}^{x_k} u \bigg)^{\frac{pr}{q(p-r)}}  V_p(a,x_i)^{\frac{pr}{p-r}} \bigg)^{\frac{p-r}{pr}}\\
& \approx 
\bigg( \sum_{k=-\infty}^{\infty} 2^{-k\frac{p}{p-r}} \bigg(\int_{x_{k-1}}^{x_k} u \bigg)^{\frac{pr}{q(p-r)}}  V_p(a,x_{k-1})^{\frac{pr}{p-r}} \bigg)^{\frac{p-r}{pr}}\lesssim \mathcal{B}_2.
\end{align*}
If $N >-\infty$, \eqref{lim<sup-2} together with \eqref{sup-int-estimate} yields, 
\begin{align*}
A_{4,3} & \lesssim \bigg( \sum_{k=N+1}^{\infty} 2^{-k\frac{p}{p-r}} \bigg(\int_{x_{k-1}}^{x_k} \bigg(\int_t^{x_k} u \bigg)^{\frac{q}{p-q}} u(t) \,  V_p(a,t)^{\frac{pq}{p-q}}dt \bigg)^{\frac{r(p-q)}{q(p-r)}} \bigg)^{\frac{p-r}{pr}} \\
& + \bigg( \sum_{k=N+2}^{\infty} 2^{-k\frac{p}{p-r}}  \sup_{N+1 \leq i < k }\bigg(\int_{x_{i}}^{x_{k}} u \bigg)^{\frac{pr}{q(p-r)}}  V_p(a,x_{i})^{\frac{pr}{p-r}} \bigg)^{\frac{p-r}{pr}}.
\end{align*}
Applying \eqref{V-cut} to the first term and \eqref{3-sum-equiv} to the second term, we have
\begin{align*}
A_{4,3} & \lesssim   \mathcal{A}_4 + \bigg( \sum_{k=N+2}^{\infty} 2^{-k\frac{p}{p-r}} \bigg(\int_{x_{k-1}}^{x_k} u \bigg)^{\frac{pr}{q(p-r)}} V_p(a,x_{k-1})^{\frac{pr}{p-r}} \bigg)^{\frac{p-r}{pr}}\lesssim \mathcal{A}_4 + \mathcal{B}_2.
\end{align*}
Thus, for any $N \in \mathbb{Z}\cup \{-\infty\}$, we arrive at $A_{4,3} \lesssim \mathcal{A}_4 + \mathcal{B}_2$. This fact, combined with \eqref{A_41<A4+B2} and \eqref{A42<A4+B2} yields $A_4 \lesssim \mathcal{A}_4 + \mathcal{B}_2$. Since $B_2 \leq \mathcal{B}_2$, we have $A_4 + B_2 \lesssim \mathcal{A}_4 + \mathcal{B}_2$ and consequently, $C\approx B_2 + A_4$. 

Now, we will show that $B_2 + A_4 \approx C_3 + C_5$.
Applying \eqref{int.equiv} with $\alpha = \frac{r}{p-r}$ and 
\begin{equation*}
h(x) = \bigg( \int_{a}^{x} \bigg(\int_{t}^{x} u \bigg)^{\frac{q}{p-q}} u(t) V_p(a,t)^{\frac{pq}{p-q}} dt \bigg)^{\frac{r(p-q)}{q(p-r)}}, \quad x\in (a, b),
\end{equation*}
it is clear that $A_4 \approx C_5$. We have also shown in \eqref{B2<C3} that $B_2 \lesssim C_3$. Hence, it remains to show that  $C_3 \lesssim A_4 + B_2$. To this end, we can use \eqref{sup-int-estimate}, and obtain $A_2 \lesssim A_4$. Moreover, we know from \eqref{C3<A2+B2} that $C_3 \lesssim A_2 + B_2$. Consequently, $C_3 \lesssim A_4+ B_2$ holds and the proof is complete. 
\qed

\

\textbf{Proof of Theorem~\ref{Cor}}
We will prove that inequality \eqref{monot.ineq-ab} holds for all $f \in \mathfrak{M}^{\uparrow}(a,b)$ if and only if inequality
\begin{equation}\label{special}
\bigg(\int_a^b \bigg(\int_a^x \bigg(\int_a^t h(\tau) d\tau \bigg)^\frac{1}{p} u(t) dt \bigg)^{q} w(x) dx  \bigg)^{\frac{p}{q}} \leq C^p \int_a^b h(x) \bigg(\int_x^b v\bigg) dx
\end{equation}
holds for all $h \in \mathfrak{M}^+(a,b)$. 

Assume that \eqref{monot.ineq-ab} holds for all $f \in \mathfrak{M}^{\uparrow}(a,b)$. Substituting $f(x)= \big(\int_a^x h\big)^{\frac{1}{p}}$, $x\in (a,b)$ for $h \in \mathfrak{M}^+(a,b)$ in \eqref{monot.ineq-ab} and applying Fubini on the right-hand side, \eqref{special} follows. 

Conversely, assume that \eqref{special} holds for all $h \in \mathfrak{M}^+(a,b)$.  Since any $f \in \mathfrak{M}^{\uparrow}(a,b)$, even if $f(0)>0$, can be approximated pointwise from below by a function of the form $f(x)^p=\int_a^x h$, $x\in (a,b)$, then  the validity of \eqref{special}  yields \eqref{monot.ineq-ab}. Therefore, the result follows from Theorem~\ref{T:main}.

\qed

\

\textbf{Acknowledgments}
T.~\"{U}nver would like to express her gratitude to the Institute of Mathematics of the Czech Academy of Sciences for hosting her and providing her a perfect working atmosphere and supplying technical support.
The research of  A.~Gogatishvili 
was partially supported by  the grant project 23-04720S of the Czech Science Foundation (GA\v{C}R),  The Institute of Mathematics, CAS is supported  by RVO:67985840  and by  Shota Rustaveli National Science Foundation (SRNSF), grant no: FR21-12353.
 The research of T.~\"{U}nver was supported by the grant of The Scientific and Technological Research Council of Turkey (TUBITAK), Grant No: 1059B192000075.

\end{document}